\documentclass[12pt,a4paper,leqno]{amsart}

\usepackage[top=1.1in, bottom=1.3in, left=1.5in, right=1.5in]{geometry}

\usepackage[ansinew]{inputenc}
\usepackage[T1]{fontenc}
\usepackage[english]{babel}
\usepackage{amsthm}
\usepackage{amsfonts}         
\usepackage{amsmath}
\usepackage{amssymb}
\usepackage{breqn}
\usepackage{bm}
\usepackage{url}
\usepackage{esint}
\usepackage{fancyhdr}

\newcommand{\R}{\mathbb{R}}

\newcommand{\N}{\mathbb{N}}

\newcommand{\Z}{\mathbb{Z}}

\newcommand{\A}{\mathcal{A}}

\newcommand{\HH}{\mathcal{H}}
\newcommand{\PP}{\mathbb{P}}
\newcommand{\LL}{\mathbb{L}}
\newcommand{\BB}{\mathbb{B}}

\theoremstyle{plain}
\newtheorem{theorem}{Theorem}
\newtheorem{lemma}[theorem]{Lemma}
\newtheorem{prop}[theorem]{Proposition}
\newtheorem{cor}[theorem]{Corollary}

\theoremstyle{remark}
\newtheorem{remark}{Remark}

\numberwithin{equation}{section}
\begin{document}

\title{Limit points of normalized prime gaps}

\author{Jori Merikoski}

\address{Department of Mathematics and Statistics, University of Turku, FI-20014 University of Turku,
Finland}
\email{jori.e.merikoski@utu.fi}

\begin{abstract} We show that at least 1/3 of positive real numbers are in the set of limit points of normalized prime gaps. More precisely, if $p_n$ denotes the $n$th prime and $\LL$ is the set of limit points of the sequence $\{(p_{n+1}-p_n)/\log p_n\}_{n=1}^\infty,$ then for all $T\geq 0$ the Lebesque measure of $\LL \cap [0,T]$ is at least $T/3.$ This improves the result of Pintz (2015) that  the Lebesque measure of $\LL \cap [0,T]$ is at least $(1/4-o(1))T,$ which was obtained by a refinement of the previous ideas of Banks, Freiberg, and Maynard (2015). Our improvement comes from using Chen's sieve to give, for a certain sum over prime pairs, a better upper bound than what can be obtained using Selberg's sieve. Even though this improvement is small, a modification of the arguments Pintz and Banks, Freiberg, and Maynard shows that this is sufficient. In addition, we show that there exists a constant $C$ such that for all $T \geq 0$ we have $\LL \cap [T,T+C] \neq \emptyset,$ that is, gaps between limit points are bounded by an absolute constant.
\end{abstract}

\maketitle
\tableofcontents

\section{Introduction and main results}
The Prime Number Theorem tells us that the gap $p_{n+1}-p_n$ between consecutive primes is asymptotically $\log p_n$ on average ($p_n$ denotes the $n$th prime). It is therefore reasonable to consider the distribution of the normalized prime gaps $(p_{n+1}-p_n)/\log p_n;$ by heuristics given by Cram\'er's model we expect that for all $b>a \geq0$
\begin{align} \label{heur}
\frac{1}{N} \bigg\{ n \leq N: \, (p_{n+1}-p_n)/\log p_n \in [a,b] \bigg\} \sim \int_a^b e^{-u} \, du, \quad \quad N \to \infty.
\end{align}
That is, we expect the sequence of normalized prime gaps to satisfy a Poisson distribution (cf. Soundararajan's account \cite{Sound} for details). Gallagher \cite{Gal} has shown this to be true assuming a sufficiently uniform version of the Hardy-Littlewood conjecture. 

To approach (\ref{heur}), consider the following conjecture of Erd\"os \cite{Erdos}: if  $\LL$ denotes the set limit points of the sequence $\{(p_{n+1}-p_n)/\log p_n\}_{n=1}^\infty,$ then $\LL = [0,\infty].$ By the 1931 result of Westzynthius \cite{West} we know that $\infty \in \LL,$ and from the seminal work of Goldston, Pintz and Y\i ld\i r\i m \cite{GPY} it follows that $0 \in \LL.$  Besides $0$ and $\infty$ no other real number is known to be in $\LL$.

It is known that $\LL$ has a positive Lebesque measure (Erd\"os \cite{Erdos} and Ricci \cite{Ricci}). Goldston and Ledoan \cite{GL} extended the method of Erd\"os to show that intervals of certain specific form, e.g. $[1/8,2]$, contain limit points. In addition, Pintz \cite{Pintz2} has shown that there is an ineffective constant $c$ such that $[0,c] \subseteq \LL$ (by applying the ground-breaking work of Zhang \cite{Zhang} on bounded gaps between primes).

Note that $\LL$ is Lebesque-measurable since it is a closed set. Hildebrand and Maier \cite{HM} showed that there exists a positive constant $c$ such that the Lebesque measure of $\LL \cap [0,T] $ is at least $ cT$ for all sufficiently large $T$.  Following the breakthrough of Maynard \cite{May} on bounded gaps between primes, it was proved by Banks, Freiberg and Maynard \cite{BFM} that this holds with $c=1/8 -o(1),$ that is, asymptotically at least 1/8 of positive real numbers are limit points. Pintz \cite{Pintz} improved this to $c=1/4 -o(1)$ by modifying the argument of \cite{BFM}; this was shown by Pintz for more general normalizations also. This was then extended to more general and especially larger normalizing factors than $\log p_n$ by Freiberg and Baker \cite{BF}, by combining the arguments with the work of Ford, Green, Konyagin, Maynard, and Tao \cite{Ford} on long prime gaps. 

For clarity we only consider the set of limit points $\LL$ with the logarithmic normalization as defined above. Our main results are deduced from the following
\begin{theorem} \label{maint}
Let $\beta_1 \leq \beta_2 \leq \beta_3 \leq \beta_4$ be any real numbers. Then
\begin{align*}
\LL \cap \{\beta_j -\beta_i: \, \, 1 \leq i < j \leq 4\} \neq \emptyset.
\end{align*}
\end{theorem}

The proof of this will be given in Section \ref{mainsec}. We note that \cite[Theorem 1.1]{BFM} gives this for nine real numbers in place of four, and \cite[Theorem 1]{Pintz} is the same but for five real numbers. Using the same argument as in the proof of \cite[Corollary 1.2]{BFM} this implies that the Lebesque measure of  $\LL \cap [0,T]$ is $\geq  (1/3-o(1))T$ as $T \to \infty$, where the $o(1)$ is ineffective. Using a more elaborate construction based on similar ideas we will show below
\begin{cor} \label{cor1} 
For all $T > 0$ we have
\begin{align*}
\mu (\LL \cap [0,T])  \geq  T/3,
\end{align*}
where $\mu$ denotes the Lebesque measure on $\R$.
\end{cor}

Another way to approach the conjecture that $\LL = [0,\infty]$ would be to show that for any given positive real $x$ we can find a limit point close to $x$; using Theorem \ref{maint}, we will show below that gaps between limit points are bounded by an absolute (ineffective) constant (note that this actually follows already from \cite[Theorem 1.1]{BFM}, as is evident from the proof):
\begin{cor} \label{cor2}
There exists a constant $C \geq 0$ such that for all $T\geq 0$ we have 
\begin{align*}
\LL \cap [T,T+C] \neq \emptyset.
\end{align*}
\end{cor}
In the language of combinatorics, a set $A \subseteq [0,\infty]$ is called syndetic if there is a constant $C$ such that every interval of length $C$ intersects with $A$ (cf. \cite{BFW}, for example). Thus, Corollary \ref{cor2} can be rephrased as saying that the set of limit points $\LL$ is syndetic.

\begin{remark} By similar ideas as in the work of Baker and Freiberg \cite{BF},  one can extend our results to other normalizations of prime gaps, replacing $\log p_n$ by a function which can grow somewhat quicker than the logarithm (cf. \cite[Theorem 6.2]{BF} for what normalizations are allowed). We have restricted our attention to the logarithmic normalization to avoid having to define cumbersome notation, with the hope that this makes the article more accessible.
\end{remark}

\subsection{Proof of Corollary \ref{cor1}} 
Corollary \ref{cor1} follows from combining Theorem \ref{maint} with the following general proposition:
\begin{prop} Let $k \geq 2$ and let $\BB \subseteq [0,\infty)$ be any Lebesque-measurable set satisfying the following property:  for any real numbers $\beta_1 \leq \beta_2 \leq \cdots \leq \beta_k$  we have
\begin{align*}
\BB \cap \{\beta_j -\beta_i: \, \, 1 \leq i < j \leq k\} \neq \emptyset.
\end{align*}
Then for any $T >0$ we have
\begin{align*}
\mu (\BB \cap [0,T])  \geq  T/(k-1).
\end{align*}
\end{prop}
\begin{proof}
For any $m \geq 2$ and for any real numbers $\beta_1,\dots,\beta_m$, define the set of differences
\begin{align*}
\Delta(\beta_1,\dots,\beta_m) := \{\beta_j -\beta_i: \, \, 1 \leq i < j \leq m\}.
\end{align*}
For any $\epsilon > 0,$ let us inductively define increasing sequences of real numbers $r_j$ and $s_j$ as follows:  set $r_0=s_0=0,$ and for $j >0$, having defined $r_0,\dots, r_{j-1}$ and $s_0,\dots,s_{j-1},$ let
\begin{align*}
S_{j}=S_j(r_0,\dots,r_{j-1}) := \{ s > r_{j-1}: \, \Delta(r_0,r_1,\dots, r_{j-1}, s) \cap \BB = \emptyset \}, \quad \quad s_{j} := \inf S_j,
\end{align*}
and pick any $r_j \in S_j$ with $r_j \in [s_j,s_j+\epsilon).$ Then by the assumption on $\BB$ for some $\ell \leq k-1$ we have $s_\ell = \infty$ (i.e. $S_\ell = \emptyset$), and we stop there and set $r_\ell = \infty$.

We now note the following property which holds for all $j \in \{0,1,\dots,\ell-1\}:$ since $s_{j+1}$ is the infimum of $S_{j+1}$, for any $t \in [r_j,s_{j+1})$  we have $\Delta(r_0,r_1,\dots, r_{j}, t) \cap \BB \neq \emptyset$. Since $\Delta(r_0,r_1,\dots, r_{j}) \cap \BB = \emptyset,$ this implies that
\begin{align*}
[r_j,s_{j+1}) \subseteq \bigcup_{i=0}^j (\BB+r_i).
\end{align*}
Hence, for all $t \in (r_j,s_{j+1}]$ we have by sub-additivity 
\begin{align} \label{tin}
\mu([r_j,t)) =\mu \bigg([r_j,t)\cap \bigcup_{i=0}^j (\BB+r_i)\bigg) \leq \sum_{i=0}^j \mu\left( [r_j,t)\cap (\BB+r_i)\right).
\end{align}

Let $T > 0.$ Then there is some $\lambda \leq \ell -1$ with $T \in (r_\lambda, r_{\lambda+1}]$ (since $r_\ell = \infty$). Denote $\tilde{T}= \min \{T,s_{\lambda+1}\}$. Then, by using $r_j < s_j+\epsilon$, we have
\begin{align*}
\mu([0,T)) = \sum_{j=0}^{\lambda-1}\mu([r_j,r_{j+1})) + \mu([r_\lambda,T)) 
\leq (\lambda+1)\epsilon +  \sum_{j=0}^{\lambda-1}\mu([r_j,s_{j+1})) + \mu([r_\lambda,\tilde{T})) .
\end{align*}
By using (\ref{tin}) for all of the summands we get
\begin{align*}
T=\mu([0,T)) & \leq (\lambda+1)\epsilon +  \sum_{j=0}^{\lambda-1}\sum_{i=0}^j  \mu( [r_j,s_{j+1})\cap (\BB+r_i)) +  \sum_{i=0}^\lambda \mu( [r_\lambda,\tilde{T})\cap (\BB+r_i))  \\
&\leq (\lambda+1)\epsilon +  \sum_{j=0}^{\lambda-1}\sum_{i=0}^j  \mu( [r_j,r_{j+1})\cap (\BB+r_i)) +  \sum_{i=0}^\lambda \mu( [r_\lambda,T)\cap (\BB+r_i))  \\
&= (\lambda+1)\epsilon + \sum_{i=0}^\lambda \bigg( \sum_{j=i}^{\lambda-1} \mu( [r_j,r_{j+1})\cap (\BB+r_i))  +   \mu( [r_\lambda,T)\cap (\BB+r_i))  \bigg)  \\
&= (\lambda+1)\epsilon + \sum_{i=0}^\lambda \mu( [r_i,T)\cap (\BB+r_i))  \\
&=  (\lambda+1)\epsilon + \sum_{i=0}^\lambda \mu( [0,T-r_i)\cap \BB)   \leq (\lambda+1)\epsilon +(\lambda+1) \mu( [0,T)\cap \BB).
\end{align*}
Hence, for any $T > 0$ we have $ \mu( [0,T)\cap \BB) \geq T/(\lambda+1) - \epsilon \geq T/(k-1) - \epsilon$ by using $\lambda+1 \leq \ell \leq k-1$. Since $\epsilon>0$ can be made arbitrarily small, we have $\mu( [0,T)\cap \BB) \geq T/(k-1).$
\end{proof}

\subsection{Proof of Corollary \ref{cor2}}
Corollary \ref{cor2} follows from Theorem \ref{maint} using the following general proposition. This is also proved in the work of Bergelson, Furstenberg, and Weiss \cite[Section 1, second paragraph]{BFW} but we give our own different proof of this.
\begin{prop} \label{gen2} Let $\BB \subseteq [0,\infty)$ be any set satisfying the following property:  there exists an integer $k \geq 2$ such that for any real numbers $\beta_1 \leq \beta_2 \leq \cdots \leq \beta_k$  we have
\begin{align*}
\BB \cap \{\beta_j -\beta_i: \, \, 1 \leq i < j \leq k\} \neq \emptyset.
\end{align*}
Then there exists a constant $C\geq 0$ (ineffective) such that for all $T\geq 0$ we have 
\begin{align*}
\BB \cap [T,T+C] \neq \emptyset.
\end{align*}
\end{prop}
 To prove this proposition we first prove the following weaker version:
\begin{lemma} \label{boundlemma} Let $\BB \subseteq [0,\infty)$ satisfy the assumptions of Proposition \ref{gen2}. Let $w$ be any given function such that $w(T) \to \infty$ as $T \to \infty,$ and $w(T) >0$ for $T>0.$  Then there exists a constant $C,$ depending only on the choice of $w,$  such that for all $T > C$ we have
\begin{align*}
\BB \cap [T-w(T),T] \neq \emptyset.
\end{align*}
\end{lemma}
\begin{proof}
Define
\begin{align*}
\A:=\{A > 0: \, \, \BB \cap [A-w(A),A] = \emptyset \}.
\end{align*}
Suppose that the conclusion of the lemma is not true, so that $\A$ is unbounded. Then we can choose $A_1, A_2, \dots, A_{k-1} \in \A$ such that
\begin{align}
& A_1  < A_2 < \cdots < A_{k-1}, \\
& w(A_1)  < w(A_2) < \cdots < w(A_{k-1}) \quad \text{and}  \label{2}\\
& A_j  < w(A_{j+1}) \quad \text{for} \quad j=1,2, \dots, k-1. \label{3}
\end{align}
Define $k$ real numbers by $\beta_0:=0$ and $\beta_j:=A_{j}$ if $j=1,2, \dots k-1.$ Then by (\ref{2}) and (\ref{3}) we have $\beta_i < w(A_j)$ if $1\leq i<j \leq k-1.$ Hence,
\begin{align*}
\{\beta_j-\beta_i: \, \, 0 \leq i < j \leq k-1\} \subseteq \bigcup_{j=1}^{k-1} [A_j-w(A_j),A_j].
\end{align*}
But by assumption we also have
\begin{align*}
\BB \cap \{\beta_j-\beta_i: \, \, 0 \leq i < j \leq k-1\} \neq \emptyset,
\end{align*}
which gives a contradiction.
\end{proof}

\emph{Proof of Proposition \ref{gen2}.}
Suppose that no such constant $C$ exists. This implies that for every $C$ there are arbitrarily large $A$ such that $\BB \cap [A-C,A]  = \emptyset$. Hence, it is possible to find a strictly increasing sequence of positive real numbers $A_n \to \infty$ as $n \to \infty,$ such that
\begin{align*}
\BB \cap [A_n -n, A_n ] = \emptyset
\end{align*}
for all $n \geq 1.$ Fix any such sequence $A_n$ and define a step function $w$ by setting  (with $A_0=0$)
\begin{align*}
w(A) = n \quad \text{for} \quad A \in (A_{n-1},A_n] \quad \text{for any} \, \, n \geq 1. 
\end{align*}
Then $w(A) \to \infty$ as $A \to \infty$, and there are arbitrarily large $A$ such that $\BB \cap [A-w(A),A] = \emptyset,$ namely $A = A_n$ for any $n\geq 1$. This is a contradiction with Lemma \ref{boundlemma}. \qed

\subsection{Outline of the proof of Theorem \ref{maint}}
The proof of Theorem \ref{maint} will occupy us for the remainder of the article; our proof builds heavily on the earlier work of Banks, Freiberg and Maynard \cite{BFM}, and the refinement of Pintz \cite{Pintz} to their argument. We now give an informal outline of the basic ideas and indicate our modifications to them.

A finite set of integers $\HH$ is said to be admissible if for every prime $p$ the set $\HH$ 
avoids at least one residue class modulo $p$, that is, if
\begin{align*}
\bigg| \bigg\{ n \,\, (p):  \,  \prod_{h \in \HH} (n+h) \equiv 0 \quad (p) \bigg\} \bigg| < p.
\end{align*}

Let $N$ be large and suppose we are given an admissible $K$-tuple $\HH = \{h_1,\dots, h_K\}$ with $h_j \leq C \log N$ for all $j$ for some large $C.$ Then by a variant of the Erd\"os-Rankin construction (cf.  \cite[Section 5]{BFM}), one can show that there is an integer $b$ and a smooth modulus $W < N^\epsilon$ such that for any $N < n \leq 2N$ with $n \equiv b \,(W),$ if there are prime numbers in the interval $[n,n+C\log N]$, then they must belong to the set $n+\HH$. 

By using the Maynard-Tao sieve, we can show that there exists $N < n \leq 2N$ with $n \equiv b \,(W)$ such that $n+\HH$ contains prime numbers once $K=|\HH|$ is large enough. Furthermore, suppose that we have a partition $\HH = \HH_1 \cup \HH_2 \cup \cdots \cup \HH_{M}$ into $M$ sets of equal size. Then we can show that there exists a constant $A$ such that for any integer $a \geq 1,$ if $M = \lceil Aa \rceil +1,$ then for at least $a+1$ distinct indices $j$ the set $n+\HH_j$ contains a prime number. That is, the prime numbers that we find by the Maynard-Tao sieve are not too much concentrated on any particular set $n+\HH_j.$ 

The constant $A$ is determined by how well we can control sums over prime pairs; more precisely, it is the best constant so that for all distinct $h,h' \in \HH$ we can show the bound
\begin{align} \label{pairsout}
\sum_{\substack{N < n \leq 2N \\ n \equiv b \, (W)}} 1_{\PP}(n+h)1_{\PP}(n+h')\bigg( \sum_{\substack{d_1, \dots, d_K \\ d_i | n+h_i}} \lambda_{d_1, \dots, d_K} \bigg)^2 \leq (A+o(1)) X,
\end{align}
where $X$ is the expected main term and $\lambda_{d_1, \dots, d_K}$ are sieve weights of Maynard-Tao type supported on $d_1\cdots d_K \leq N^\delta$ for some small $\delta >0$. In \cite[Section 4]{BFM}, Selberg's upper bound sieve is used to show this for $A=4.$ We improve this to $A=3.99$ by using Chen's sieve \cite{Chen}, \cite{Pan} (cf. Proposition \ref{pairs} below). 

The reason why this small improvement is sufficient is as follows: we choose $a=100$ so that $\lceil 3.99a \rceil +1 = 4a,$ and partition our tuple
\begin{align*}
\HH = \HH_1 \cup \HH_2 \cup \HH_3 \cup \HH_4, \quad \quad \quad \quad \HH_i = \bigcup_{j=1}^a \HH_{ij}, \quad i \in \{1,2,3,4\}.
\end{align*}
Then we find $N < n \leq 2N$ with $n \equiv b \,(W)$ such that for at least $a+1$ distinct $(i,j)$ the set $n+\HH_{ij}$ contains a prime number. Thus, by the pigeon-hole principle we must have at least two indices $i \neq i'$ such that both $n+\HH_i, n+\HH_{i'}$ contain primes. By the restriction $n\equiv b \, (W)$ given by the modified Erd\"os-Rankin construction, we then know that there are two consecutive primes, one in $n+\HH_i$ and one in $n+\HH_{i'},$ for some $i \neq i'.$ For $\beta_1 \leq \beta_2 \leq \beta_3 \leq \beta_4$  as in Theorem \ref{maint}, it is then enough to choose $\HH_i$ so that for all $h \in \HH_i$ we have $h = (\beta_i +  o(1)) \log N.$ From this argument we see that the exact numerical value of $A=3.99$ is not important, what matters is that $A$ is strictly less than $4.$ 

To show the bound (\ref{pairsout}) with $A=3.99,$ we require a Bombieri-Vinogradov type equidistribution result for primes, where the moduli run over multiples of $W < N^\epsilon.$ The possibility of exceptional zeros of $L$-functions causes some technical problems, but the result \cite[Theorem 4.2]{BFM} turns out to be sufficient. Since we are using Chen's sieve, we also need to extend this to almost-primes; this is done in Section \ref{bvsec}.  In Section \ref{chensec} we apply Chen's sieve to obtain the required bound (\ref{pairsout}) for prime pairs (Proposition \ref{pairs}). We then state and prove in Section \ref{maysec} the precise version of the Maynard-Tao sieve which we will use (Proposition \ref{strong}), and in Section \ref{mainsec} we prove our main result Theorem \ref{maint}.

\begin{remark}
By the same argument, if we could show the bound (\ref{pairsout}) with any constant $A <3$ in place of $3.99,$ we would obtain Theorem \ref{maint} with sequence of four real numbers replaced by three. This in turn would give that $\mu(\LL \cap [0,T]) \geq T/2.$ Similarly, if we had (\ref{pairsout}) with any constant $A < 2$ in place of $3.99,$ we could show that $\LL =[0,\infty],$ which is the conjecture of Erd\"os. However, by the parity principle this should be just as hard as obtaining a lower bound for such a sum over prime pairs, which would immediately imply $\LL =[0,\infty]$ (cf. \cite[Chapter 16]{FI} for a quantitative version due to Bombieri of the parity principle).
\end{remark}

\subsection{Notations}
We use the following asymptotic notations: for positive functions $f,g,$ we write $f \ll g$ or $f= \mathcal{O}(g)$ if there is a constant $C$ such that $f  \leq C g.$ $f \asymp g$ means $g \ll f \ll g.$ The constant may depend on some parameter, which is indicated in the subscript (e.g. $\ll_{\epsilon}$).
We write $f=o(g)$ if $f/g \to 0$ for large values of the variable.

In general, $C$ stands for some large constant, which may not be the same from place to place. For variables we write $n \sim N$ meaning $N<n \leq eN$ (an $e$-adic interval), and $n \asymp N$ meaning $N/C < n < CN$ (a $C^2$-adic interval) for some constant $C>1$ which is large enough depending on the situation. If not otherwise stated the symbols $p,q,r$ denote primes and $d,k,\ell,m,n$ denote integers.

For a statement $E$ we denote by $1_E$ the characteristic function of that statement. For a set $A$ we use $1_A$ to denote the characteristic function of $A,$ so that $1_\PP$ will denote the characteristic function of primes.

We define $P(w):= \prod_{p\leq w} p,$ and for any integer $d$ we write $P^-(d):= \min \{p: \, p | d\},$ $P^+(d):= \max \{p: \, p | d\}.$ The $k$-fold divisor function is denoted by $\tau_k(d).$ We denote  the ceiling function by $\lceil \cdot \rceil$, that is, $\lceil x \rceil $ is the smallest integer $n \geq x.$

Overall we use similar notations as in \cite{BFM}, especially when we use the Maynard-Tao sieve; these are recalled in the text as needed.

\subsection*{Acknowledgements}  I am grateful to my supervisor Kaisa Matom\"aki for support and comments. I also express my gratitude to Emmanuel Kowalski for helpful comments as well as for hospitality during my visit to ETH Z\"urich. I wish to thank James Maynard for bringing the article \cite{BFM} to my attention. I also wish to thank Pavel Zorin-Kranich for useful suggestions and the anonymous referee for comments. During the work the author was supported by a grant from the Magnus Ehrnrooth Foundation.

\section{Modified Bombieri-Vinogradov Theorem} \label{bvsec}
As was outlined above, we need to show an upper bound of type (\ref{pairsout}) for prime pairs, where the modulus $W$ can be as large as $N^\epsilon.$ For this purpose we require a modified version of the Bombieri-Vinogradov Theorem. Before stating this we need the following lemma on exceptional zeros of Dirichlet $L$-functions (this is \cite[Lemma 4.1]{BFM}):
\begin{lemma}Let $T \geq 3$ and $P \geq T^{1/\log_2 T}.$ For a sufficiently small constant $c > 0,$ there is at most one modulus $q \leq T$ with $P^+(q) \leq P$ and one  primitive character $\chi$ modulo $q$ such that the function $L(s,\chi)$ has a zero in the region
\begin{align*}
\Re(s) \geq 1-\frac{c}{\log P}, \quad \quad |\Im(s)| \leq \exp \bigg( \log P / \sqrt{\log T}\bigg). 
\end{align*}
If such a character $\chi$ mod $q$ exists, then it is real, $L(s,\chi)$ has at most one zero in the above region, which is then real and simple, and
\begin{align*}
P^+(q) \gg \log q \gg \log_2 T.
\end{align*}
\end{lemma}

Fix a constant $c>0$ for which the above lemma holds. Similarly as in \cite{BFM}, if such an exceptional modulus $q \leq T$ exists with $P = T^{1/\log_2 T}$, we define
\begin{align} \label{Z}
Z_T = P^+(q),
\end{align}
and we set $Z_T =1$ if no such modulus exists. We then have the following variant of the Bombieri-Vinogradov Theorem (this is \cite[Theorem 4.2]{BFM}): 
\begin{prop} \label{bv} \emph{\textbf{(Modified Bombieri-Vinogradov).}} Let $N > 2$ and fix constants $C > 0,$ $\epsilon > 0,$ and $\delta > 0$. Let $q_0 < N^{\epsilon}$ be a square-free integer with $P^+(q_0) < N^{\epsilon/ \log_2 N}.$ Then for $\epsilon$ small enough we have
\begin{align*}
\sum_{\substack{q \leq N^{1/2-\delta} \\ q_0 | q \\ (q, Z_{N^{2\epsilon}})=1}} \max_{(a,q) = 1} \bigg | \sum_{\substack{ n \leq N \\ n \equiv a \, (q)}} \Lambda (n) \, - \frac{1}{\phi(q)}  \sum_{\substack{ n \leq N}} \Lambda (n)\bigg | \, \ll_{\delta, C} \, \frac{N}{\phi(q_0) \log^C N}.
\end{align*}
\end{prop}

From the proof of \cite[Theorem 4.2]{BFM} we obtain the following lemma, which we require for the proof of Proposition \ref{bv2} below:

\begin{lemma} \label{char} With the same notations and assumptions as in Proposition \ref{bv} we have
\begin{align} 
\sup_{\substack{A,B \\ AB \leq N^{1/2-\delta} \\ A \leq q_0}} \sum_{\substack{A < a \leq 2A \\ a | q_0}} \sum_{\substack{B \leq b \leq 2B \\ (b, q_0 Z_{N^{2\epsilon}})=1}} \frac{1}{\phi (b)} \sideset{}{'}\sum_{\chi \, \, (ab)} \bigg | \sum_{n \leq N} \Lambda(n) \chi(n) \bigg | \, \ll_C \frac{N}{\log^C N},
\end{align}
where $\Sigma '$ denotes the sum over primitive characters modulo $ab.$
\end{lemma}

Since we plan to apply Chen's sieve, we also require a similar equidistribution result for almost-primes. To prove such a result we require the large sieve for multiplicative characters, which follows from Theorem 9.10 of \cite{FI}:
\begin{lemma} \label{large}\emph{\textbf{(Large sieve for multiplicative characters).}} For any sequence $c_n$ of complex numbers and for any $M,N \geq 1$ we have
\begin{align*}
\sum_{q \leq Q} \frac{q}{\phi(q)} \sideset{}{'} \sum_{\chi \, \, (q)} \bigg | \sum_{M< n \leq M+N} c_n \chi(n)\bigg |^2 \, \leq \, (Q^2+N) \sum_n |c_n|^2.
\end{align*}
\end{lemma}

To state the equidistribution result for almost-primes, we need to set up some notation: fix $0 < \alpha < 1/2,$ and let $N^\alpha \ll A_1 \ll N^{1-\alpha}$, for sufficiently large $N.$ Define
\begin{align} \label{p0}
\Lambda_0(n) := (f \ast g)(n),
\end{align}
where $f(m) = 1_{\PP}(m)(\log m)1_{m \leq A_1},$ and $g$ is any function such that $|g(n)| \, \ll 1,$ and $g(n) \neq 0$ only if $P^-(n) \geq N^{\alpha}$ and $n \asymp N/A_1$. Note that then $\Lambda_0(n)$ is supported on almost-primes $n \ll N$.

We then have that Proposition \ref{bv} holds also with $\Lambda(n)$ replaced by $\Lambda_0 (n)$: 
\begin{prop} \label{bv2}\emph{\textbf{(Modified Bombieri-Vinogradov for almost-primes).}}  Let $N > 2$ and fix constants $C > 0,$ $\epsilon > 0,$ and $\delta > 0$.   Let $q_0 < N^{\epsilon}$ be a square-free integer with $P^+(q_0) < N^{\epsilon/ \log_2 N}.$ Let $\Lambda_0(n)$ be as in (\ref{p0}). Then for all small enough $\epsilon$ we have
\begin{align*}
\sum_{\substack{q \leq N^{1/2-\delta} \\ q_0 | q \\ (q, Z_{N^{2\epsilon}})=1}} \max_{(a,q) = 1} \bigg | \sum_{\substack{  n \equiv a \, (q)}} \Lambda_0 (n) \, - \frac{1}{\phi(q)}  \sum_{\substack{ n }} \Lambda_0 (n)\bigg| \, \ll_{\delta, C} \, \frac{N}{\phi(q_0) \log^C N}.
\end{align*}
\end{prop}
\begin{proof} The basic idea is to use the large sieve inequality for large moduli and for small moduli use Lemma \ref{char}. For convenience we set $D:=N^{1/2-\delta}.$ Using the expansion
\begin{align*}
\sum_{\substack{  n \equiv a \, (q)}} \Lambda_0 (n) \, - \frac{1}{\phi(q)}  \sum_{\substack{ n }} \Lambda_0 (n) = \frac{1}{\phi(q)}\sum_{\substack{\chi \, \, (q) \\ \chi \neq \chi_0}} \overline{\chi}(a) \sum_{n } \Lambda_0 (n) \chi (n),
\end{align*}
we are reduced to obtaining the bound
\begin{align} \label{ave}
\sum_{\substack{q \leq D \\ q_0 | q \\ (q, Z_{N^{2\epsilon}})=1}}  \frac{1}{\phi(q)} \sum_{\substack{\chi \, \, (q) \\ \chi \neq \chi_0}} \bigg |  \sum_{n } \Lambda_0 (n) \chi (n) \bigg| \, \ll_{\delta, C} \, \frac{N}{\phi(q_0) \log^C N}.
\end{align}
We then replace the character $\chi$ modulo $q$ by the primitive character $\chi'$ modulo $q'$ which induces $\chi$; we have
\begin{align*}
\chi(n) = \chi'(n) - \chi'(n)1_{(n,q/q') > 1}.
\end{align*}
Hence, the left-hand side of (\ref{ave}) is bounded by
\begin{align} \label{ave2}
\sum_{\substack{q \leq D \\ q_0 | q \\ (q, Z_{N^{2\epsilon}})=1}}  \frac{1}{\phi(q)} \sum_{\substack{\chi \, \, (q) \\ \chi \neq \chi_0}} \bigg |  \sum_{n } \Lambda_0 (n) \chi' (n) \bigg| + \sum_{\substack{q \leq D \\ q_0 | q \\ (q, Z_{N^{2\epsilon}})=1}}  \frac{1}{\phi(q)} \sum_{\substack{\chi \, \, (q) \\ \chi \neq \chi_0}} \bigg |  \sum_{m,n } f(m) g(n) \chi' (mn) 1_{(mn,q/q') > 1} \bigg|.
\end{align}
We have
\begin{align*}
1_{(mn,q/q') > 1} = 1_{(n,q/q') > 1} + 1_{(m,q/q') > 1} -1_{(m,q/q') > 1}1_{(n,q/q') > 1}.
\end{align*}
Define $h_1(n,d) := 1$ and $h_2(n,d) := 1_{(n,d)>1}.$ Then (\ref{ave2}) is bounded by
\begin{align} \label{ave3}
\sum_{i,j=1}^2 \sum_{\substack{q \leq D \\ q_0 | q \\ (q, Z_{N^{2\epsilon}})=1}}  \frac{1}{\phi(q)} \sum_{\substack{\chi \, \, (q) \\ \chi \neq \chi_0}} \bigg |  \sum_{m,n } f(m)h_i(m,q/q')\chi'(m) g(n)h_j(n,q/q') \chi' (n)  \bigg|.
\end{align}
If $i =2$ or $j=2,$ we remove the additional conditions for $q$, write $q=dq'$, and bound the sum by
\begin{align*}
 \sum_{\substack{q \leq D}}  \frac{1}{\phi(q)} \sum_{\substack{\chi \, \, (q) \\ \chi \neq \chi_0}} &\bigg |  \sum_{m,n } f(m)h_i(m,q/q')\chi'(m) g(n)h_j(n,q/q') \chi' (n)  \bigg| \\
& \ll  \sum_{d \leq D} \frac{1}{\phi(d)} \sum_{\substack{q' \leq D}}  \frac{1}{\phi(q')} \sideset{}{'}\sum_{\substack{\chi \, \, (q')}} \bigg |  \sum_{m,n } f(m)h_i(m,d)\chi(m) g(n)h_j(n,d) \chi (n)  \bigg| \\
& \ll (\log N) \sum_{d \leq D} \frac{1}{\phi(d)} \sup_{E \leq D} \frac{1}{E} \bigg(  \sum_{\substack{q' \sim E}}  \frac{q'}{\phi(q')} \sideset{}{'} \sum_{\substack{\chi \, \, (q')}} \bigg |  \sum_{m } f(m)h_i(m,d)\chi(m)  \bigg|^2 \bigg)^{1/2} \\ 
& \hspace{80pt}  \cdot\bigg(  \sum_{\substack{q' \sim E}}   \frac{q'}{\phi(q')} \sideset{}{'}\sum_{\substack{\chi \, \, (q')}} \bigg |  \sum_{n }  g(n)h_j(n,d) \chi (n)  \bigg|^2 \bigg)^{1/2},
\end{align*}
where in the last bound we have split the sum over $q'$ dyadically and applied Cauchy-Schwarz. By Lemma \ref{large} and by the assumptions on $f$ and $g,$ the last expression is bounded by
\begin{align} \label{j2}
(\log N)\sum_{d \leq D} \frac{1}{\phi(d)} \sup_{E \leq D} \frac{1}{E}  &\bigg(  \bigg(E^2 + A_1 \bigg) \sum_{m } |f(m)h_i(m,d)|^2 \bigg)^{1/2}   \\ \nonumber
& \hspace{20pt}  \cdot\ \bigg(  \bigg(E^2 + N/A_1 \bigg) \sum_{n } |g(n)h_j(n,d)|^2  \bigg)^{1/2}
\end{align}
Suppose at first that $j=2$ so that $h_j(n,d)= 1_{(n,d)>1}.$ Since $g(n)$ is supported on $P^-(n) \geq N^\alpha,$ this means that $(n,d) \geq N^\alpha.$ We obtain that (\ref{j2}) is bounded by
\begin{align*}
(\log N)&\sum_{d \leq D} \frac{1}{\phi(d)} \sup_{E \leq D} \frac{1}{E}  \bigg(  \bigg(E^2 + A_1 \bigg) \sum_{m } |f(m)|^2 \bigg)^{1/2}   \\ 
& \hspace{150pt}  \cdot\ \bigg(  \bigg(E^2 + N/A_1 \bigg) \sum_{\substack{k| d \\ N^{\alpha} \leq k \leq D}}  \sum_{n \asymp N/(A_1k)} |g(kn)|^2  \bigg)^{1/2} \\
& \ll (\log^2 N) \sum_{d \leq D} \frac{\tau(d)^{1/2}}{\phi(d)} \sup_{E \leq D} \frac{1}{E}  \bigg( \bigg(E^2 + A_1 \bigg) A_1 \bigg)^{1/2} \bigg( \bigg(E^2 + N/A_1 \bigg) N^{1-\alpha}/A_1 \bigg)^{1/2} \\
& \leq  (\log^4 N) \sup_{E \leq D}  ( E N^{(1-\alpha)/2}  + N^{1-\alpha/2}/A_1 + A_1 + N^{1-\alpha/2} /E) \ll N^{1- \alpha/3},
\end{align*}
which is sufficient. For $i=2, j=1,$ since $f(m)=1_{\PP}(m)(\log m)1_{n \leq A_1},$ we have that if $(m,d)>1,$ then $m$ is a prime dividing  $d$. Hence, by a similar argument as above we get a bound $\ll  N^{1- \alpha/3}$.

For $i=j=1$ we have to estimate
\begin{align} \label{j1}
 \sum_{\substack{q \leq D \\ q_0 | q \\ (q, Z_{N^{2\epsilon}})=1}}  \frac{1}{\phi(q)} \sum_{\substack{\chi \, \, (q) \\ \chi \neq \chi_0}} \bigg |  \sum_{m,n } f(m)\chi'(m) g(n) \chi' (n)  \bigg|.
\end{align}
We begin by extracting a factor of $1/\phi(q_0)$ similarly as in the proof of \cite[Theorem 4.2]{BFM}: if $q'$ denotes the modulus of $\chi'$, then (\ref{j1}) is bounded by (writing  $q'=ab,$ where $a | q_0$ and $(b,q_0)=1$; recall that $q_0$ is square-free)
\begin{align*}
\sum_{\substack{q' \leq D \\ (q', Z_{N^{2\epsilon}})=1}}   \sideset{}{'}\sum_{\substack{\chi \, \, (q') }} &\bigg |  \sum_{m,n } f(m)\chi(m) g(n) \chi(n)  \bigg|  \sum_{\substack{q \leq D \\ [q',q_0] | q \\ (q, Z_{N^{2\epsilon}})=1}} \frac{1}{\phi(q)} \\
& \ll \frac{\log N}{\phi (q_0)} \sum_{a | q_0} \sum_{\substack{b \leq D/a \\ (b, q_0Z_{N^{2\epsilon}})=1}}  \frac{1}{\phi (b)} \sideset{}{'}\sum_{\substack{\chi \, \, (ab) }} \bigg |  \sum_{m,n } f(m)\chi(m) g(n) \chi (n)  \bigg| \\
& \ll \frac{\log^3 N}{\phi (q_0)} \sup_{\substack{A,B \\ AB \leq D \\ A \leq q_0}} \sum_{\substack{A < a \leq 2A \\ a | q_0}} \sum_{\substack{B \leq b \leq 2B \\ (b, q_0 Z_{N^{2\epsilon}})=1}} \frac{1}{\phi (b)} \sideset{}{'}\sum_{\substack{\chi \, \, (ab) }} \bigg |  \sum_{m,n } f(m)\chi(m) g(n) \chi (n)  \bigg| \\
\end{align*}
Hence, it remains to show that
\begin{align*}
\sup_{\substack{A,B \\ AB \leq D \\ A \leq q_0}} \sum_{\substack{A < a \leq 2A \\ a | q_0}} \sum_{\substack{B \leq b \leq 2B \\ (b, q_0 Z_{N^{2\epsilon}})=1}} \frac{1}{\phi (b)} \sideset{}{'}\sum_{\substack{\chi \, \, (ab) }} \bigg |  \sum_{m,n } f(m)\chi(m) g(n) \chi (n)  \bigg| \, \ll_C \frac{N}{\log^C N}
\end{align*}

 For $B \geq N^{\epsilon}$ we have by Cauchy-Schwarz and Lemma \ref{large}
\begin{align*}
& \sum_{\substack{A < a \leq 2A \\ a | q_0}} \sum_{\substack{B \leq b \leq 2B \\ (b, q_0 Z_{N^{2\epsilon}})=1}} \frac{1}{\phi (b)} \sideset{}{'}\sum_{\substack{\chi \, \, (ab) }} \bigg |  \sum_{m,n } f(m)\chi(m) g(n) \chi (n)  \bigg|  \\
& \ll \frac{1}{B} \bigg( \sum_{q \ll AB} \frac{q}{\phi(q)} \sideset{}{'}\sum_{\substack{\chi \, \, (q) }} \bigg |  \sum_{m } f(m)\chi(m) \bigg|^2\bigg)^{1/2} \bigg(\sum_{q \ll AB} \frac{q}{\phi(q)} \sideset{}{'}\sum_{\substack{\chi \, \, (q) }} \bigg |  \sum_{n } g(n) \chi (n)  \bigg|^2   \bigg)^{1/2} \\
& \ll  \frac{\log N}{B} \bigg((AB)^2 A_1 + A_1^2 \bigg)^{1/2} \bigg( (AB)^2 N/A_1 + (N/A_1)^2 \bigg)^{1/2} \\
& \leq (\log N) ( A^2 B N^{1/2}  + A N / A_1^{1/2} +  A_1^{1/2} A N^{1/2} + N/B ) \ll N^{1-\epsilon},
\end{align*}
if $\epsilon$ is small enough in terms of $\delta$ and $\alpha.$

For $B < N^{\epsilon}$  we replace  $f(m)$ by $\Lambda(m)1_{m \leq A_1},$ which causes an error term bounded by using a trivial bound 
\begin{align*}
\sum_{\substack{A < a \leq 2A \\ a | q_0}} \sum_{\substack{B \leq b \leq 2B \\ (b, q_0 Z_{N^{2\epsilon}})=1}} \frac{1}{\phi (b)} \sideset{}{'}\sum_{\substack{\chi \, \, (ab) }} &\bigg |  \sum_{\substack{p^k \leq A_1 \\ k \geq 2}} \sum_{n } \log(p) \chi(p^k) g(n) \chi (n)  \bigg| \\
 & \hspace{20pt} \ll (AB)^2 N^{1-\alpha/2} \log N \ll N^{1-\alpha/3} 
\end{align*}
if $\epsilon$ is sufficiently small. We then use Cauchy-Schwarz to get
\begin{align} \nonumber
\sum_{\substack{A < a \leq 2A \\ a | q_0}} \sum_{\substack{B \leq b \leq 2B \\ (b, q_0 Z_{N^{2\epsilon}})=1}} & \frac{1}{\phi (b)} \sideset{}{'}\sum_{\substack{\chi \, \, (ab) }} \bigg |  \sum_{m,n } \Lambda(m)1_{m \leq A_1} \chi(m) g(n) \chi (n)  \bigg|  \\ \label{smallb}
& \ll  \bigg( \sum_{\substack{A < a \leq 2A \\ a | q_0}} \sum_{\substack{B \leq b \leq 2B \\ (b, q_0 Z_{N^{2\epsilon}})=1}} \frac{1}{\phi (b)} \sideset{}{'}\sum_{\substack{\chi \, \, (ab) }} \bigg |  \sum_{m \leq A_1 } \Lambda (m) \chi(m) \bigg|^2\bigg)^{1/2} \\ \nonumber
& \hspace{140pt} \cdot \bigg(\frac{1}{B}\sum_{q \ll AB} \frac{q}{\phi(q)} \sideset{}{'}\sum_{\substack{\chi \, \, (q) }} \bigg |  \sum_{n } g(n) \chi (n)  \bigg|^2   \bigg)^{1/2} 
\end{align}
Since $AB < N^{2\epsilon} < A_1^{1/2-\delta},$ we may use the bound Lemma \ref{char} with $A_1$ in place of $N$ (decreasing $\epsilon$ also if necessary), which yields
\begin{align*}
\sum_{\substack{A < a \leq 2A \\ a | q_0}}& \sum_{\substack{B \leq b \leq 2B \\ (b, q_0 Z_{N^{2\epsilon}})=1}} \frac{1}{\phi (b)} \sideset{}{'}\sum_{\substack{\chi \, \, (ab) }} \bigg |  \sum_{ m \leq A_1 } \Lambda(m)\chi(m) \bigg|^2 \\
& \ll A_1  \sum_{\substack{A < a \leq 2A \\ a | q_0}} \sum_{\substack{B \leq b \leq 2B \\ (b, q_0 Z_{N^{2\epsilon}})=1}} \frac{1}{\phi (b)} \sideset{}{'}\sum_{\substack{\chi \, \, (ab) }} \bigg |  \sum_{m \leq A_1 } \Lambda(m)\chi(m) \bigg| \ll_C \frac{A_1^2}{\log^{2(C+5)} N}.
\end{align*}
Using Lemma \ref{large} to bound the sum with $g(n) \chi (n)$ in (\ref{smallb}) we get that 
\begin{align*}
\sum_{\substack{A < a \leq 2A \\ a | q_0}} & \sum_{\substack{B \leq b \leq 2B \\ (b, q_0 Z_{N^{2\epsilon}})=1}} \frac{1}{\phi (b)} \sideset{}{'}\sum_{\substack{\chi \, \, (ab) }} \bigg |  \sum_{m,n } f(m)\chi(m) g(n) \chi (n)  \bigg| \\
& \ll_C \frac{A_1}{\log^{C+5} N}  \bigg( A^2 B N/A_1 + (N/A_1)^2/B \bigg)^{1/2} \ll_C \frac{N}{\log^{C+5} N}.
\end{align*}
\end{proof}

\section{Chen's sieve upper bound for prime pairs} \label{chensec}
In this section we will apply Chen's sieve to obtain an upper bound for prime pairs, which is 3.99 times the expected main term. As will become apparent in the next section, the exact numerical value of this constant does not matter, only that it is stricly less than four. To state the result, we first need to set up some notation from \cite{BFM}.

Let $K>1,$ $N > 3,$ and define the Maynard-Tao sieve weights (recall the definition of $Z_T$ from (\ref{Z}))
\begin{align} \label{l1}
\lambda_{d_1, \dots, d_K} = \begin{cases} \bigg(\prod_{i=1}^K \mu (d_i)\bigg) \sum_{j=1}^{J} \prod_{\ell=1}^K F_{\ell,j}\bigg( \frac{\log d_\ell}{\log N}\bigg), & \text{if} \, \, (d_1\cdots d_K, Z_{N^{4 \epsilon}}) =1, \\
0, & \text{otherwise,} \end{cases}
\end{align}
for some fixed $J$, where $F_{\ell,j}:[0,\infty) \to \R$ are smooth compactly supported functions, not identically zero, satisfying a support condition
\begin{align} \label{l2}
\sup \bigg \{\sum_{\ell=1}^K t_l: \, \, \prod_{\ell=1}^K F_{\ell,j}(t_\ell) \neq 0 \bigg \} \leq \delta
\end{align}
for all $j=1,2,\dots,J$ for some small $\delta >0.$ Note that this implies that $\lambda_{d_1,\dots, d_K}$ are supported on $d_1 \cdots d_K \leq N^{\delta}.$ Define
\begin{align*}
F(t_1, \dots, t_K) := \sum_{j=1}^J \prod_{\ell=1}^K F_{\ell,j}'\bigg( t_\ell \bigg),
\end{align*}
where $F_{\ell,j}'$ is the derivative of $F_{\ell,j}.$ Set
\begin{align} \label{Lint}
L_K(F)& := \int_0^\infty \cdots \int_0^\infty \bigg( \int_0^\infty \int_0^\infty F(t_1,\dots t_K) dt_{K-1} dt_K\bigg)^2 dt_1 \cdots dt_{K-2} \\ \nonumber
 & =   \sum_{j,j'=1}^J F_{K-1,j}(0)F_{K-1,j'}(0)F_{K,j}(0)F_{K,j'}(0) \prod_{\ell=1}^{K-2} \int_0^\infty F'_{\ell,j}(t_\ell)F'_{\ell,j'}(t_\ell) dt_\ell.
\end{align}
 We note here that  $F_{\ell,j}$ will be chosen so that $F(t_1,\dots,t_K)$ is symmetric with respect to permutations of the variables (cf. \cite{BFM}). Let $Z_{N^{4\epsilon}}$ be as in (\ref{Z}) and define
\begin{align*}
W := \prod_{\substack{p \leq \epsilon \log N \\ p \nmid Z_{N^{4\epsilon}}}} p, \quad  \quad \quad \quad B := \frac{\phi(W)}{W} \log N.
\end{align*}

Using the above notation, we have that \cite[Lemma 4.6 (iii)]{BFM} holds with the constant $4$ replaced by $3.99:$
\begin{prop} \label{pairs} For all  sufficiently large $N$ the following holds: 

Let $\HH = \{h_1, \dots, h_K\} \subseteq [0,N]$ be an admissible $K$-tuple such that 
\begin{align} \label{smooth}
P^+\bigg (\prod_{1 \leq i < j\leq K} (h_j-h_i) \bigg) \leq \epsilon \log N.
\end{align}
Let $b$ be an integer such that 
\begin{align*}
\bigg( \prod_{j=1}^K (b+h_j) , W\bigg) =1.
\end{align*}
Then for all distinct $h_j,h_\ell \in \HH$ we have
\begin{align*}
S:=\sum_{\substack{N < n \leq 2N \\ n \equiv b \, (W)}} 1_{\PP}(n+h_j)1_{\PP}(n+h_\ell)\bigg( \sum_{\substack{d_1, \dots, d_K \\ d_i | n+h_i}} \lambda_{d_1, \dots, d_K} \bigg)^2 \leq (3.99 + \mathcal{O}(\delta)) \frac{N}{W} B^{-K}L_K(F).
\end{align*}
\end{prop}

The proof in \cite[Lemma 4.6 (iii)]{BFM} uses Selberg's sieve combined with the Modified Bombieri-Vinogradov Theorem. Our improvement comes from using Chen's sieve instead of Selberg's sieve. Similarly as in \cite[Lemma 4.6 (iii)]{BFM}, we first note that we may replace
\begin{align*}
\bigg( \sum_{\substack{d_1, \dots d_K \\ d_i | n+h_i}} \lambda_{d_1, \dots, d_K} \bigg)^2 \quad \text{by} \quad
\nu_{\HH,j,\ell} (n) := \bigg( \sum_{\substack{d_1, \dots, d_K \\ d_i | n+h_i \\ d_j=d_\ell=1}} \lambda_{d_1, \dots, d_K} \bigg)^2 1_{((n+h_j)(n+h_\ell),Z_{N^{4\epsilon}})=1} 
\end{align*}
in the sum $S.$

We then require the following weighted sieve inequality of Chen type (this is essentially Lemma 4.1 of \cite{Wu}, which is in there attributed to Chen \cite{Chen}; according to Wu, the idea that this simple sieve inequality is sufficient is due to Pan \cite{Pan}).
\begin{lemma} \label{chen} Let $0 < \alpha < \beta < 1/4,$ $Y:= N^\alpha,$ and $Z:=N^{\beta}.$
Then $S \leq S_1 -S_2 /2 + S_3 /2,$ where
\begin{align*}
S_1 & := \sum_{\substack{N < n \leq 2N \\ n \equiv b \, (W)}} 1_{\PP}(n+h_j)1_{(n+h_\ell,P(Y))=1}\nu_{\HH,j,\ell} (n)  \\
S_2 & := \sum_{Y < p \leq Z} \sum_{\substack{N < n \leq 2N \\ n \equiv b \, (W) \\ p | n+h_\ell}} 1_{\PP}(n+h_j)1_{(n+h_\ell,P(Y))=1}\nu_{\HH,j,\ell} (n)  , \quad \quad \text{and} \\
S_3 &:= \sum_{\substack{N < n \leq 2N \\ n \equiv b \, (W)}} 1_{\PP}(n+h_j) \sum_{Y < p < q < r \leq Z} \sum_{(s,P(q))=1} 1_{n+h_\ell=pqrs}\nu_{\HH,j,\ell} (n) .
\end{align*}
\end{lemma}
\begin{proof}
By positivity of $\nu_{\HH,j,\ell} (n)$ it suffices to show that for any $n \in (N+h_\ell,2N+h_\ell]$
\begin{align} \label{sieve}
1_{(n,P(Z))=1} \leq 1_{(n,P(Y)) = 1} - \frac{1}{2} \sum_{Y < p \leq Z} 1_{p|n} 1_{(n,P(Y)) = 1} +\frac{1}{2} \sum_{Y < p < q < r \leq Z} \sum_{(s,P(q))=1} 1_{n=pqrs}.
\end{align}
For $(n,P(Y))>1$ this is obvious, so let $(n,P(Y))=1$ and denote $k=\sum_{Y < p \leq Z} 1_{p|n}.$ If $k=0,$ then both sides of (\ref{sieve}) are equal to one. For $k \geq 1$ the left-hand side is zero. If $k=1,$ then the right-hand side is $1-1/2+0 = 1/2 > 0.$ For $k \geq 2$ the right-hand side is $1-k/2 + (k-2)/2=0,$ since in the last sum $p$ and $q$ are fixed and there are $k-2$ ways to choose $r$.
\end{proof}

\begin{remark} Note that $\beta < 1/4$ implies that in the sum $S_3$ we have $s \gg N/(pqr) > N^{1/4} > q.$ The above lemma holds also for $\beta \geq 1/4,$ but then we sometimes may have $s=1$ in the sum $S_3.$
\end{remark}

We now proceed to estimate $S_1$, $S_2$ and $S_3$ separately by applying the linear sieve. For this we use similar notations as in  \cite[Chapters 11 and 12]{FI} (using the subscript `lin' for clarity): we let $F_{\text{lin}}(s),f_{\text{lin}}(s)$ be the continuous solution to the system of delay-differential equations
\begin{align*}
\begin{cases} (sF_{\text{lin}}(s))' = f_{\text{lin}}(s-1) \\
(sf_{\text{lin}}(s))' = F_{\text{lin}}(s-1)
\end{cases}
\end{align*}
with the condition
\begin{align*}
\begin{cases} sF_{\text{lin}}(s) = 2e^{\gamma}, & \text{if} \, \, 1 \leq s \leq 3 \\
sf_{\text{lin}}(s) = 0, & \text{if} \, \, s\leq 2.
\end{cases}.
\end{align*}
Here $\gamma$ is the Euler-Mascheroni constant. We record here that for $2 \leq s \leq 4$
\begin{align*}
f_{\text{lin}}(s) = \frac{2 e^\gamma \log (s-1)}{s}.
\end{align*}
By \cite[Chapters 11 and 12]{FI} we then have
\begin{lemma}\label{linear} \emph{\textbf{(Linear sieve).}}
Let $(a_n)_{n \geq 1}$ be a sequence of non-negative real numbers. For some fixed $X$ depending only on the sequence $(a_n)_{n \geq 1}$, define $r_d$ for all square-free $d \geq 1$ by
\begin{align*}
\sum_{n \equiv 0 \, (d)} a_n = g(d) X + r_d,
\end{align*}
where $g(d)$ is a multiplicative function, depending only on the sequence $(a_n)_{n \geq 1}$, satisfying $0 \leq g(p) < 1$ for all primes $p.$ Let $D\geq 2$ (the level of distribution), and let $z=D^{1/s}$ for some $s\geq 1.$ Suppose that there exists a constant $L >0$ that for any $2 \leq w < z$ we have
\begin{align*}
\prod_{w \leq p < z} (1-g(p))^{-1} \leq \frac{\log z}{\log w} \bigg(1+\frac{L}{\log w}\bigg).
\end{align*}
Then
\begin{align*}
\sum_{n} a_n 1_{(n,P(z))=1} &\leq (F_{\text{\emph{lin}}}(s) + \mathcal{O}(\log^{-1/6} D)) X \prod_{p\leq z} (1-g(p)) + \sum_{\substack{d \leq D \\ d \, \, \text{\emph{squarefree}}}} |r_d|, \\
\sum_{n} a_n 1_{(n,P(z))=1} &\geq (f_{\text{\emph{lin}}}(s) - \mathcal{O}(\log^{-1/6} D)) X \prod_{p\leq z} (1-g(p)) - \sum_{\substack{d \leq D \\ d \, \, \text{\emph{squarefree}}}} |r_d|.
\end{align*}
\end{lemma}

We now estimate the sums $S_1,S_2$ and $S_3$ in the following three lemmata.
\begin{lemma} \label{s1} We have
\begin{align*}
S_1 \leq \frac{F_{\text{\emph{lin}}}(1/(2\alpha)) + \mathcal{O}(\delta)}{\alpha e^\gamma}  \frac{N}{W} B^{-K}L_K(F)
\end{align*}
\end{lemma}
\begin{proof}
Define $r_d$ by the equation
\begin{align} \label{rd}
 \sum_{\substack{N < n \leq 2N \\ n \equiv -h_\ell \, (d)}} 1_{\PP}(n+h_j) 1_{n \equiv b \, (W)}\nu_{\HH,j,\ell} (n)  = g(d)  \sum_{\substack{N < n \leq 2N}} 1_{\PP}(n+h_j) 1_{n \equiv b \, (W)} \nu_{\HH,j,\ell} (n)  +r_d,
\end{align}
where $g(d)$ is a multiplicative function, supported on square-free integers, defined by
\begin{align*}
g(p) := \begin{cases} \frac{1}{p-1}, & \text{if} \,  p \, \nmid W Z_{N^{4\epsilon}} \\
0, & \text{if} \,  p \, \mid W Z_{N^{4\epsilon}}.
\end{cases}
\end{align*}
We note that by the same argument as in the proof of \cite[Lemma 4.6]{BFM} (recall that $d_j=d_\ell=1$ in $\nu_{\HH,j,\ell} (n) $), the sum on the right-hand side in (\ref{rd}) is
\begin{align} \nonumber
\sum_{\substack{N < n \leq 2N}} 1_{\PP}(n+h_j) 1_{n \equiv b \, (W)} \nu_{\HH,j,\ell} (n)  &= (1+o(1)) \frac{N }{\phi(W) \log N}  B^{-K+2}L_K(F) \\ \label{main}
&= (1+o(1)) \frac{N}{W} B^{-K+1}L_K(F).
\end{align}
 (to show this we just expand the square in $\nu_{\HH,j,\ell} (n)$, swap the order of summation, and use the Proposition \ref{bv} with moduli $[d_1,d_1'] \cdots [d_K,d_K']W \leq N^{3\delta}$ similarly as in \cite[Lemma 4.6]{BFM}). Hence, by the upper bound of the linear sieve  (Lemma \ref{linear} with level of distribution $D=N^{1/2-4\delta}$, sifting up to $Y=N^\alpha$) we get
\begin{align*}
S_1 \leq (F_{\text{lin}}(1/(2\alpha)) + \mathcal{O}(\delta))\bigg( \prod_{p \leq Y} (1- g(p)) \bigg) \frac{N}{W} B^{-K+1 }L_K(F) + \sum_{\substack {d \leq N^{1/2-4\delta} \\ d \, \, \text{squarefree}} } |r_d|.
\end{align*}
By Merten's Theorem
\begin{align*}
 \prod_{p \leq Y} (1- g(p)) &= \prod_{W < p < Y}\bigg( 1- \frac{1}{p-1} \bigg) = \prod_{W < p < Y}\bigg( 1- \frac{1+ \mathcal{O}(1/p)}{p} \bigg) \\
 & = (1 + o(1))\frac{W}{\phi(W)}  \prod_{ p < Y}\bigg( 1- \frac{1}{p} \bigg) = (1+o(1))\frac{W}{\phi(W) e^{\gamma} \log Y} ,
\end{align*}
so that
\begin{align*}
S_1 \leq \frac{F_{\text{lin}}(1/(2\alpha)) + \mathcal{O}(\delta)}{\alpha e^\gamma}  \frac{N}{W} B^{-K}L_K(F)+ \sum_{\substack {d \leq N^{1/2-4\delta} \\ d \, \, \text{squarefree}} } |r_d|.
\end{align*}
For the error term we expand the square in $\nu_{\HH,j,\ell} (n) $ and swap the order of summation to get
\begin{align*}
r_d &=  \sum_{\substack{N < n \leq 2N \\ n \equiv -h_\ell \, (d)}} 1_{\PP}(n+h_j) 1_{n \equiv b \, (W)} \nu_{\HH,j,\ell} (n)  - g(d)  \sum_{\substack{N < n \leq 2N}} 1_{\PP}(n+h_j) 1_{n \equiv b \, (W)} \nu_{\HH,j,\ell} (n)  \\
&= \sum_{\substack{d_1, \dots, d_K \\ d'_1,\dots d_K' \\ d_j=d_j'=d_\ell=d_\ell'=1}} \lambda_{d_1,\dots, d_K}  \lambda_{d_1',\dots, d_K'} \bigg( \sum_{\substack{N < n \leq 2N \\ n \equiv b \, (W) \\n \equiv -h_\ell \, (d) \\ n \equiv -h_i \, ([d_i,d_i'])}} 1_{\PP}(n+h_j) - g(d) \sum_{\substack{N < n \leq 2N \\ n \equiv b \, (W) \\ n \equiv -h_i \, ([d_i,d_i'])}} 1_{\PP}(n+h_j) \bigg).
\end{align*}
Similarly as in the proof of \cite[Lemma 4.6]{BFM}, we note that since $h'-h$ is $\epsilon \log N$-smooth for all distinct $h,h' \in \HH$ by (\ref{smooth}), and by the support conditions (\ref{l1}), (\ref{l2}) of $\lambda_{d_1,\dots,d_k},$ we may assume that $d,$ $[d_1,d_1'],\dots, [d_K,d_K'],$ $W Z_{N^{4\epsilon}}$ are pairwise coprime. In that case we have $g(d)=1/\phi(d)$,
\begin{align*}
\sum_{\substack{N < n \leq 2N \\ n \equiv b \, (W) \\n \equiv -h_\ell \, (d) \\ n \equiv -h_i \, ([d_i,d_i'])}} 1_{\PP}(n+h_j) = \frac{\pi(2N+h_j) - \pi (N+h_j)}{\phi(d) \phi(W) \prod_{i=1}^K \phi([d_i,d_i'])} + \mathcal{O} \bigg( E(N, d [d_1,d_1'] \cdots [d_K, d_K'] W) \bigg), 
\end{align*}
and
\begin{align*}
 g(d) \hspace{-5pt}\sum_{\substack{N < n \leq 2N \\ n \equiv b \, (W) \\ n \equiv -h_i \, ([d_i,d_i'])}} 1_{\PP}(n+h_j)   = \frac{\pi(2N+h_j) - \pi (N+h_j)}{\phi(d) \phi(W) \prod_{i=1}^K \phi([d_i,d_i'])} + \mathcal{O} \bigg(  E(N, [d_1,d_1'] \cdots [d_K, d_K'] W) \bigg)
\end{align*}
where
\begin{align*}
E(N,q) = \max_{(a,q)=1 } \bigg |  \pi (2N+h_j; q, a ) - \pi(N+h_j;q,a) - \frac{\pi(2N+h_j) - \pi (N+h_j)}{\phi(q)} \bigg |,
\end{align*}
if $(q, Z_{N^{4\epsilon}}) =1 $ and we set $E(N,q) = 0$ if $(q, Z_{N^{4\epsilon}}) >1.$

Hence, by the triangle inequality
\begin{align*}
 \sum_{\substack {d \leq N^{1/2-4\delta} \\ d \,\, \text{squarefree}} } |r_d| \, \ll \sum_{\substack {d \leq N^{1/2-4\delta} \\ d \,\, \text{squarefree} \\ (d,W)=1} } \sum_{\substack{d_1, \dots, d_K \\ d'_1,\dots d_K' \\ d_j=d_j'=d_\ell=d_\ell'=1 }} | \lambda_{d_1,\dots, d_K}  \lambda_{d_1',\dots, d_K'} | E(N, d [d_1,d_1'] \cdots [d_K, d_K'] W) \\
 + \sum_{\substack {d \leq N^{1/2-4\delta} \\ d \,\, \text{squarefree} \\ (d,W)=1}}  \frac{1}{\phi(d)}\sum_{\substack{d_1, \dots, d_K \\ d'_1,\dots d_K' \\ d_j=d_j'=d_\ell=d_\ell'=1}} | \lambda_{d_1,\dots, d_K}  \lambda_{d_1',\dots, d_K'} |   E(N, [d_1,d_1'] \cdots [d_K, d_K'] W)
\end{align*}
The second sum on the right-hand side is bounded by $\log N$ times the first sum. We have the trivial bounds $|\lambda_{d_1,\dots, d_K}| \, \ll 1$ and $E(N,q) \ll 1 +N/\phi(q).$  Hence, using Cauchy-Schwarz and Proposition \ref{bv}  the first sum is bounded by
\begin{align*}
\sum_{\substack{q \leq N^{1/2-2\delta} \\ (q, W Z_{N^{4\epsilon}}) =1 }} & \tau_{3K}(q) E(N, qW) \\
&\leq  \bigg( \sum_{\substack{q \leq N^{1/2-2\delta} \\ (q, W Z_{N^{4\epsilon}}) =1 }}\tau_{3K}(q)^2 (1+N /\phi(qW))\bigg)^{1/2}\bigg( \sum_{\substack{q \leq N^{1/2-2\delta} \\ (q, W Z_{N^{4\epsilon}}) =1 }}  E(N, qW) \bigg)^{1/2} \\
& \, \ll_{K,C} \frac{N}{W \log^C N},
\end{align*}
which is sufficient.
\end{proof}

\begin{lemma} \label{s2} We have
\begin{align*}
S_2 \geq  \frac{1-\mathcal{O}(\delta)}{\alpha e^{\gamma}} \int_\alpha ^\beta f_{\text{\emph{lin}}} \bigg( \frac{1/2 -t}{\alpha} \bigg) \frac{dt}{t} \frac{N}{W} B^{-K}L_K(F).
\end{align*}
\end{lemma}
\begin{proof}
Set
\begin{align*}
S_{2,p} := \sum_{\substack{N < n \leq 2N \\ n \equiv b \, (W) \\ p | n+h_\ell}} 1_{\PP}(n+h_j)1_{(n+h_\ell,P(Y))=1}\nu_{\HH,j,\ell} (n),
\end{align*}
so that $S_2 = \sum_{Y < p \leq Z} S_{2,p}.$ We will apply the lower bound of the linear sieve to each of the sums $S_{2,p}:$ for $(d,p)=1,$ let $r_{dp}$ be defined by
\begin{align*}
 \sum_{\substack{N < n \leq 2N \\ n \equiv b \, (W) \\ p | n+h_\ell \\ n \equiv - h_\ell \, (d)}} 1_{\PP}(n+h_j) 1_{n \equiv b \, (W)}\nu_{\HH,j,\ell} (n)  = \frac{g(d)}{p-1}  \sum_{\substack{N < n \leq 2N}} 1_{\PP}(n+h_j) 1_{n \equiv b \, (W)} \nu_{\HH,j,\ell} (n)  +r_{dp},
\end{align*}
where $g(d)$ is as in the proof of Lemma \ref{s1}, that is, a multiplicative function, supported on square-free integers, defined by
\begin{align*}
g(q) := \begin{cases} \frac{1}{q-1}, & \text{if} \,  q \, \nmid W Z_{N^{4\epsilon}} \\
0, & \text{if} \,  q \, \mid W Z_{N^{4\epsilon}}.
\end{cases}
\end{align*}
Applying the lower bound of the linear sieve (Lemma \ref{linear} with level of distribution $D=N^{1/2-4\delta}/p$ and shifting up to $Y=N^\alpha$), using (\ref{main}) and Merten's Theorem similarly as in the proof of Lemma \ref{s1}, we find that
\begin{align*}
S_{2,p} & \geq \bigg( f_{\text{lin}} \bigg( \frac{\log N^{1/2}/p}{\log Y} \bigg) - \mathcal{O}(\delta)\bigg)\frac{1}{p-1} \bigg( \prod_{q \leq Y} (1- g(q)) \bigg) \frac{N}{W} B^{-K+1}L_K(F) - \sum_{\substack {d \leq N^{1/2-4\delta}/p \\ d \, \, \text{squarefree}} } |r_{dp}| \\
& \geq  \frac{1}{\alpha e^\gamma} \bigg( f_{\text{lin}} \bigg( \frac{\log N^{1/2}/p}{\log Y} \bigg) - \mathcal{O}(\delta)\bigg)\frac{1}{ p} \frac{N}{W} B^{-K}L_K(F) - \sum_{\substack {d \leq N^{1/2-4\delta}/p \\ d \, \, \text{squarefree}} } |r_{dp}|.
\end{align*}
Summing over $p$ we get, by a similar argument as in the proof of Lemma \ref{s1}, a sufficient bound for the error term
\begin{align*}
\sum_{Y < p \leq Z} \sum_{\substack {d \leq N^{1/2-4\delta}/p \\ d \, \, \text{squarefree}} } |r_{dp}| \, \ll_{C,K} \frac{N}{W\log^C N}.
\end{align*}
Hence, we have
\begin{align*}
S_{2} & \geq \frac{1- \mathcal{O}(\delta)}{\alpha e^\gamma} \bigg( \sum_{Y< p \leq Z} \frac{1}{p} f_{\text{lin}} \bigg( \frac{\log N^{1/2}/p}{\log Y} \bigg) \bigg) \frac{N}{W} B^{-K}L_K(F) \\
 & \geq \frac{1- \mathcal{O}(\delta)}{\alpha e^\gamma} \bigg( \int_{Y < z \leq Z}  f_{\text{lin}} \bigg( \frac{\log N^{1/2}/z}{\log Y} \bigg) \frac{dz}{z \log z} \bigg)  \frac{N}{W} B^{-K}L_K(F) \\
& \geq  \frac{1-\mathcal{O}(\delta)}{\alpha e^{\gamma}} \int_\alpha ^\beta f_{\text{lin}} \bigg( \frac{1/2 -t}{\alpha} \bigg) \frac{dt}{t} \frac{N}{W} B^{-K}L_K(F)
\end{align*}
by the change of variables $z=N^t$.
\end{proof}

For the next Lemma we need the Buchstab function, defined as the continuous solution to the delay-differential equation
\begin{align*}
\begin{cases} s \omega(s) = 1, & \text{if} \, \, 1 \leq s \leq 2,\\
(s \omega (s))' = \omega(s-1), & \text{if} \, \, s > 2.
\end{cases}
\end{align*}
Then by \cite[Lemma 12.1]{FI} for any $N^\epsilon < z < N$ we have
\begin{align} \label{buchstabfun}
\sum_{N< n \leq 2N} 1_{(n,P(z))=1} = (1+o(1)) \omega(\log N / \log z) \frac{N}{\log z}, \quad \quad N \to \infty.
\end{align}

\begin{lemma} \label{s3}
We have
\begin{align*}
S_3 \leq (4 + \mathcal{O}(\delta))\int_{\alpha < u_1 < u_2 < u_3 < \beta} \omega \bigg(\frac{1-u_1-u_2-u_3}{u_2} \bigg)\frac{du_1 d u_2 du_3 }{u_1 u_2^2 u_3} \frac{N}{W} B^{-K}L_K(F).
\end{align*}
\end{lemma}
\begin{proof}
Here we apply the switching, to sieve out the prime divisors of $n+h_j$ rather than $n+h_\ell$; define
\begin{align*}
a_n := \sum_{Y < p< q < r \leq Z} \sum_{(s,P(q)) = 1} 1_{n=pqrs}
\end{align*}
so that
\begin{align*}
S_3 = \sum_{\substack{N < n \leq 2N \\ n \equiv b \, (W)}} 1_\PP(n+h_j) a_{n+h_\ell} \nu_{\HH,j,\ell} (n).
\end{align*}
We use a similar Selberg upper bound sieve as in \cite[Lemma 4.6]{BFM} (we could just as well use the linear sieve upper bound as in the above but the argument is slightly simpler this way); let $G:[0, \infty) \to \R$ be a smooth function supported on $[0,1/4-2\delta]$ with $G(0)=1.$ Then
\begin{align*}
S_3 &\leq \sum_{\substack{N < n \leq 2N \\ n \equiv b \, (W)}} a_{n+h_\ell} \bigg( \sum_{e \,| n+ h_j}\mu(e) G \bigg( \frac{\log e}{\log N}\bigg)\bigg)^2\nu_{\HH,j,\ell} (n) \\
&\leq  \sum_{\substack{N < n \leq 2N \\ n \equiv b \, (W)}} a_{n+h_\ell} \bigg( \sum_{\substack{e \,| n+ h_j\\ (e, Z_{N^{4\epsilon}})=1}}\mu(e) G \bigg( \frac{\log e}{\log N}\bigg)\bigg)^2  \bigg( \sum_{\substack{d_1, \dots, d_K \\ d_i | n+h_i \\ d_j=d_\ell=1}} \lambda_{d_1, \dots, d_K} \bigg)^2.
\end{align*}
We then expand the squares and rearrange the sum to get
\begin{align*}
 \sum_{\substack{d_1, \dots, d_K \\ d'_1,\dots d_K' \\ d_j=d_j'=d_\ell=d_\ell'=1 }}  \lambda_{d_1,\dots, d_K}  \lambda_{d_1',\dots, d_K'}   \sum_{\substack{e,e' \\ (ee', Z_{N^{4\epsilon}})=1}} \mu(e) \mu(e')  G \bigg( \frac{\log e}{\log N}\bigg) G \bigg( \frac{\log e'}{\log N}\bigg)  \sum_{\substack{N < n \leq 2N \\ n \equiv b \, (W) \\ [d_i,d_i'] | n+h_i \\ [e,e'] | n+h_j }} a_{n+h_\ell}
\end{align*}
In the innermost sum, we may again assume that  $[d_1,d_1'],\dots, [d_K,d_K'],$ $[e,e'],$ $W Z_{N^{4\epsilon}}$  are pairwise coprime, and insert the estimates (for $d=[d_1,d'_1]\cdots [d_K,d_K'][e,e'] W$)
\begin{align*}
 \sum_{\substack{N < n \leq 2N \\ n \equiv a \, (d) }} a_{n+h_\ell} = \frac{1}{\phi(d)} \sum_{\substack{N < n \leq 2N  }} a_{n+h_\ell}  + \tilde{r}_d.
\end{align*}
  By essentially the same argument as in the proof of \cite[Lemma 4.6 (iii)]{BFM}, choosing the function $G$ optimally gives 
\begin{align} \label{third}
S_3 &\leq (4+ \mathcal{O}(\delta)) \frac{\log N}{N} \bigg(\sum_{N < n \leq 2N} a_n \bigg) \frac{N}{W} B^{-K}L_K(F) + \mathcal{O}(R),
\end{align}
where 
\begin{align*}
R =    \sum_{\substack{d_1, \dots, d_K \\ d'_1,\dots d_K' \\ d_j=d_j'=d_\ell=d_\ell'=1 }} | \lambda_{d_1,\dots, d_K}  \lambda_{d_1',\dots, d_K'} | \sum_{\substack{e,e' \leq N^{1/4-2\delta} \\ (ee',Z_{N^{4\epsilon}}) = 1}}E_0(N, [d_1,d_1'] \cdots [d_K, d_K'] [e,e'] W)
\end{align*}
with
\begin{align*}
E_0(N,d) := \max_{(a,d)=1} \bigg | \sum_{\substack{N +h_\ell< n \leq 2N+h_\ell \\ n \equiv a \, (d)}} a_n -\frac{1}{\phi(d)}  \sum_{\substack{N+h_\ell < n \leq 2N+h_\ell }} a_n  \bigg |.
\end{align*}
Note that the condition $e,e' \leq N^{1/4-2\delta}$ comes from the support restriction of the function $G$.  Using Cauchy-Schwarz and the trivial bound $ | \lambda_{d_1,\dots, d_K}| \ll 1$ similarly as in the proof of Lemma \ref{s1}, the error term $R$ has a sufficient bound if we can show that
\begin{align*}
\sum_{\substack{d \leq N^{1/2-2\delta} \\ (d, W Z_{N^{4\epsilon}}) =1 }} |E_0(N,dW)| \, \ll_C \frac{N}{W \log^C N}.
\end{align*}
To show this we use finer-than-dyadic decomposition to write $a_n 1_{N +h_\ell < n \leq 2N + h_\ell}$ as a sum of terms of the form
\begin{align*}
\sum_{\substack{Y < p < q < r \leq Z \\ p \in I_1,  \, \, q \in I_2}}\, \, \sum_{\substack{ (N+h_\ell)/(pqr) < s \leq (2N+h_\ell)/(pqr) \\ (s,P(q))=1}} 1_{n=pqrs},
\end{align*}
where each $I_j$ is of the form $(A_j, \lambda A_j ]$ for $\lambda= 1 + \log^{-2C} N$. We remove the cross-conditions $Y < p < q;$ this causes an error bounded using triangle inequality by the sum of (\ref{error1}) and (\ref{error2}), which are given by
\begin{align} \label{error1}
\sum_{\substack{d \leq N^{1/2-2\delta} \\ (d, W Z_{N^{4\epsilon}}) =1 }} &  \max_{(a,d)=1} \sum_{\substack{Y < p < q < r \leq Z \\ p \in [\lambda^{-2}Y, \lambda^2 Y] \cup [\lambda^{-2}q, \lambda^2 q] \\ (pq,d) =1 }} \, \, \sum_{\substack{s \asymp N/(pqr) \\ (s,(P(q)))=1 \\ rs \equiv a \overline{pq} \, (dW)}} 1   \, \\ \nonumber
& \ll   \sum_{\substack{d \leq N^{1/2-2\delta} \\ (d, W Z_{N^{4\epsilon}}) =1 }} \max_{(a,d)=1} \sum_{\substack{Y < p < q  \leq Z \\ p \in [\lambda^{-2}Y, \lambda^2 Y] \cup [\lambda^{-2}q, \lambda^2 q] \\ (pq,d) =1 }} \, \, \sum_{\substack{m \asymp N/(pq) \\ m \equiv a \overline{pq} \, (dW)}} 1  \ll_C \frac{N}{W \log^C N}
\end{align}
(since $m=rs \gg N/pq > N^{1/2}$ by using $\beta < 1/4$), and
\begin{align}\label{error2}
\sum_{\substack{d \leq N^{1/2-2\delta} \\ (d, W Z_{N^{4\epsilon}}) =1 }} \frac{1}{\phi(dW)} \sum_{\substack{Y < p < q < r \leq Z \\ p \in [\lambda^{-2}Y, \lambda^2 Y] \cup [\lambda^{-2}q, \lambda^2 q]}} \, \, \sum_{\substack{s \asymp N/(pqr) \\ (s,(P(q)))=1}} 1 \, \ll_C \frac{N}{W \log^C N},
\end{align}
which is sufficient. Similarly, if we replace the condition $N+h_\ell < pqrs \leq 2N+h_\ell$ by  $ (N+h_\ell)/(A_1qr) < s \leq (2N+h_\ell)/(A_1qr),$ then  we get a sufficient bound for the contribution of the part where $pqrs \notin (N+h_\ell,2N+h_\ell].$  Thus, we can replace $a_n 1_{N < n \leq 2N}$ by a sum of $\mathcal{O}(\log^{4C+2} N)$ functions of the form $(P\ast g)(n), $ where for $Y \ll A_1,  A_2 \ll Z$
\begin{align*}
P(m)= 1_\PP(m)1_{m \in (A_1,\lambda A_1]} \quad \text{and} \quad
g(n)= \sum_{\substack{ q < r \leq Z \\ q \in (A_2,\lambda A_2] }} \sum_{\substack{ (N+h_\ell)/(A_1qr) < s \leq (2N+h_\ell)/(A_1qr)\\ (s,P(q))=1}} 1_{n=qrs}.
\end{align*}
We can then replace $P(m)$ by $f(m)/\log A_1,$ where  $f(m) := P(m)\log m$; this is because for all $m \in (A_1,\lambda A_1]$ we have
\begin{align*}
\log m = \log A_1 + \mathcal{O}\bigg( \log^{-2C} N \bigg),
\end{align*}
so that the error term from this has a sufficient bound by trivial estimates. Finally, writing $f(m) = 1_\PP(m)(\log m )1_{m \leq \lambda A_1} - 1_\PP(m)(\log m) 1_{m \leq A_1}$ and using triangle inequality,  we obtain by Proposition \ref{bv2} that 
\begin{align*}
\sum_{\substack{d \leq N^{1/2-2\delta} \\ (d, W Z_{N^{4\epsilon}}) =1 }} |E_0(N,dW)| \, \ll_C \frac{N}{W \log^C N},
\end{align*}
which suffices by the previous remarks to bound the error term $R$ in (\ref{third}).

To compute the main term in (\ref{third}) we write by using (\ref{buchstabfun})
\begin{align*}
\sum_{N < n \leq 2N} a_n &= \sum_{Y < p< q < r \leq Z} \, \,\sum_{ \substack{ N/(pqr) < s \leq 2N/(pqr) \\ (s,P(q)) = 1}} 1 \\
& = (1+o(1)) N \sum_{Y < p< q < r \leq Z} \frac{\omega \bigg( \frac{\log(N/(pqr))}{\log q}\bigg)}{pqr \log q} \\
&= ( 1+o(1)) N \int_{Y<z_1 < z_2 < z_3 \leq Z} \omega \bigg( \frac{\log(N/(z_1z_2z_3))}{\log z_2}\bigg)\frac{dz_1 dz_2 dz_3}{z_1 z_2 z_3 (\log z_1)( \log^2 z_2) \log z_3} \\
&= ( 1+o(1)) \frac{N}{\log N} \int_{\alpha < u_1 < u_2 < u_3 < \beta} \omega \bigg(\frac{1-u_1-u_2-u_3}{u_2} \bigg)\frac{du_1 d u_2 du_3 }{u_1 u_2^2 u_3}
\end{align*}
after the change of variables $z_j=N^{u_j}.$
\end{proof}

\emph{Proof of Proposition \ref{pairs}.} Combining Lemmata \ref{chen}, \ref{s1}, \ref{s2} and \ref{s3} we obtain
\begin{align*}
S \leq (\Omega_1 - \Omega_2 + \Omega_3 + \mathcal{O}(\delta)) \frac{N}{W} B^{-K}L_K(F),
\end{align*}
where
\begin{align*}
\Omega_1 &= \frac{F_{\text{lin}}(1/(2\alpha))}{\alpha e^\gamma}, \quad \quad \quad \Omega_2 =  \frac{1}{2\alpha e^{\gamma}} \int_\alpha ^\beta f_{\text{lin}} \bigg( \frac{1/2 -t}{\alpha} \bigg) \frac{dt}{t}, \quad  \quad \text{and} \\
\Omega_3 &= 2 \int_{\alpha < u_1 < u_2 < u_3 < \beta} \omega \bigg(\frac{1-u_1-u_2-u_3}{u_2} \bigg)\frac{du_1 d u_2 du_3 }{u_1 u_2^2 u_3}.
\end{align*}
We choose $\alpha = 1/7$ and $\beta= 3/14$ (so that $(1/2-t)/\alpha \geq 2$ in the integral defining $\Omega_2$). For this choice we get
\begin{align*}
\Omega_1 = \frac{7F_{\text{lin}}(7/2)}{e^\gamma} =   2\bigg( \frac{3F_{\text{lin}}(3)}{e^\gamma} + \int_3^{7/2} \frac{f_{\text{lin}}(s-1)}{e^\gamma} ds \bigg) \\
= 4+ 4\int_3^{7/2} \frac{\log (s-2)}{s-1} ds \leq 4.19,
\end{align*}
\begin{align*}
\Omega_2 = \frac{7}{2e^{\gamma}} \int_{1/7}^{3/14} f_{\text{lin}} \bigg( 7/2-7t \bigg) \frac{dt}{t} = 7 \int_{1/7}^{3/14}\frac{\log(7/2-7t-1)}{7/2-7t}\frac{dt}{t} \geq  0.279,
\end{align*}
and
\begin{align*}
\Omega_3 =  2 \int_{1/7 < u_1 < u_2 < u_3 < 3/14} \omega \bigg(\frac{1-u_1-u_2-u_3}{u_2} \bigg)\frac{du_1 d u_2 du_3 }{u_1 u_2^2 u_3} \leq 0.076.
\end{align*}
Hence, $\Omega_1 - \Omega_2 + \Omega_3 < 3.99.$ \qed

\begin{remark} The upper bound for the integral in $\Omega_3$ was  computed using Python 7.3; the code is available at \url{http://codepad.org/2emT1dHN}. The choice of exponents $\alpha=1/7$ and $\beta=3/14$ has not been optimized since this is not relevant to our application.
\end{remark}

\section{Modified Maynard-Tao sieve} \label{maysec}
We are now ready to prove the following version of the Maynard-Tao sieve, which is modelled after \cite[Theorem 4.3]{BFM}:
\begin{prop} \label{strong} \emph{\textbf{(Modified Maynard-Tao sieve).}} Let $K$ be a sufficiently large multiple of $4.$ Let $\epsilon > 0$ be sufficiently small. Then for all sufficiently large $N$ the following holds: 

Let $Z_{N^{4\epsilon}}$ be as in (\ref{Z}) and define
\begin{align*}
W := \prod_{\substack{p \leq \epsilon \log N \\ p \, \nmid Z_{N^{4\epsilon}}}} p;
\end{align*}
Let $\HH = \{h_1, \dots, h_K\} \subseteq [0,N]$ be an admissible $K$-tuple such that
\begin{align*}
P^+ \bigg( \prod_{1 \leq i < j\leq K} (h_j-h_i) \bigg) \leq \epsilon \log N
\end{align*}
Let $b$ be an integer such that 
\begin{align*}
\bigg( \prod_{j=1}^K (b+h_j) , W\bigg) =1.
\end{align*}
Let
\begin{align*}
\HH = \HH_1 \cup \HH_2 \cup \HH_3 \cup \HH_{4}
\end{align*}
be a partition of $\HH$ into four sets of equal size. Then there is an integer $n \in [N,2N]$ with $n \equiv b \, (W)$ such that $n+\HH_i$ contains a prime number for at least two distinct indices $i \in \{1,2,3,4\}.$
\end{prop}

To prove the above proposition we will show that it suffices to prove the following seemingly weaker
\begin{prop} \label{weak}  Let $a \geq 1$ be an integer and let $K$ be a sufficiently large multiple of $\lceil 3.99a \rceil +1.$ Let $\epsilon > 0$ be sufficiently small. Then for all sufficiently large $N$ the following holds: 

Let $Z_{N^{4\epsilon}}$ be as in (\ref{Z}) and define
\begin{align*}
W := \prod_{\substack{p \leq \epsilon \log N \\ p \nmid Z_{N^{4\epsilon}}}} p.
\end{align*}
Let $\HH = \{h_1, \dots, h_K\} \subseteq [0,N]$ be an admissible $K$-tuple such that 
\begin{align*}
P^+ \bigg( \prod_{1 \leq i < j\leq K} (h_j-h_i) \bigg) \leq \epsilon \log N
\end{align*}
Let $b$ be an integer such that 
\begin{align*}
\bigg( \prod_{j=1}^K (b+h_j) , W\bigg) =1.
\end{align*}
Let
\begin{align*}
\HH = \HH_1 \cup \HH_2 \cup \cdots \cup \HH_{\lceil 3.99a \rceil+1}
\end{align*}
be a partition of $\HH$ into $\lceil 3.99a \rceil+1$ sets of equal size. Then there is an integer $n \in [N,2N]$ with $n \equiv b \, (W)$ and a set of $a+1$ distinct indices $\{j_1,j_2,\dots, j_{a+1}\} \subseteq \{1,2,\dots,\lceil 3.99a \rceil+1\}$ such that $n+\HH_j$ contains a prime number for every $j \in \{j_1,j_2,\dots, j_{a+1}\}.$
\end{prop}

\emph{Proof of Proposition \ref{strong} using Proposition \ref{weak}.}
We take $a=100$  so that  $\lceil 3.99a \rceil+1 =4a.$ By taking a larger $K$ if necessary, we may suppose that $K$ is a sufficiently large multiple of $4a$. Given a partition $\HH = \HH_1 \cup \HH_2 \cup \HH_3 \cup \HH_{4}$ as in Proposition \ref{strong}, we take a further partition
\begin{align*}
\HH_i = \HH_{i1} \cup \HH_{i2} \cup \cdots \cup \HH_{ia}
\end{align*}
into sets of equal sizes for all $i \in \{1,2,3,4\}.$ Then by Proposition \ref{weak} there is an integer $n \in [N,2N]$ with $n \equiv b \, (W)$ so that for at least $a+1$ distinct sets $\HH_{ij}$ the set $n+ \HH_{ij}$ contains a prime number.  By the pigeon-hole principle this implies that $n+\HH_i$ contains a prime number for at least two distinct indices $i \in \{1,2,3,4\}.$ \qed
\vspace{7pt}

\emph{Proof of Proposition \ref{weak}.}
We use Pintz's refined version of the argument in \cite{BFM} (cf. proof of \cite[Theorem 3]{Pintz} and especially \cite[Theorem 5.4]{BF}): using the notations of \cite{BF}, let us denote  $M:=\lceil 3.99a \rceil+1$, and let $\mu, \mu'$ be positive real numbers with (defining $\binom{1}{2}=0$)
\begin{align} \label{mu}
\mu' = \max_{v \in \N} \bigg(v- \mu \binom{v}{2} \bigg).
\end{align}
For any integer $n$ consider
\begin{align} \label{quant}
\sum_{j=1}^M \bigg( \sum_{h \in \HH_j} 1_\PP(n+h) - \mu \sum_{\substack{\{h,h'\} \subseteq \HH_j \\ h \neq h'}} 1_\PP(n+h)1_\PP(n+h') \bigg).
\end{align}
If there are at most $a$ indices $j$ such that $n+\HH_j$ contains a prime, then the sum (\ref{quant}) is at most $\mu' a$. Hence, if
\begin{align*}
 \sum_{h \in \HH} 1_\PP(n+h)-\mu' a - \mu \sum_{j=1}^M \sum_{\substack{\{h,h'\} \subseteq \HH_j \\ h \neq h'}} 1_\PP(n+h)1_\PP(n+h') \, > \,0,
\end{align*}
then there are at least $a+1$ distinct indices $j$ such that $n+\HH_j$ contains a prime. Therefore, the proposition follows once we show that
\begin{align*}
\sum_{\substack{N<n\leq 2N \\ n \equiv b \, (W)}}\bigg( \sum_{h \in \HH} 1_\PP(n+h)-\mu' a - \mu \sum_{j=1}^M \sum_{\substack{\{h,h'\} \subseteq \HH_j \\ h \neq h'}} 1_\PP(n+h)1_\PP(n+h')\bigg) \bigg(\sum_{\substack{d_1, \dots, d_K \\ d_i | n+h_i}} \lambda_{d_1, \dots, d_K} \bigg)^2 \, > \,0.
\end{align*}
Let $\Sigma$ denote the above sum. Using \cite[Lemma 4.6 (i),(ii)]{BFM} to evaluate the first two sums, and Proposition \ref{pairs} to bound the third,  we obtain that $\Sigma$ is bounded from below by  
\begin{align*}
(1+\mathcal{O}(\delta))\frac{N}{WB^K} \bigg( K J_K(F) - \mu' a I_K(F) - 3.99 \mu M \binom{K/M}{2} L_K(F)  \bigg),
\end{align*}
where $I_K(F)$, $J_K(F)$ and $L_K(F)$ are the integrals in \cite[Lemma 4.6]{BFM} ($L_K(F)$ is the same as in (\ref{Lint}) above). By \cite[Lemma 4.7]{BFM}, for any given $\rho \in (0,1)$ there is a choice of $F$ such that
\begin{align*}
J_K(F) &\geq (1+ \mathcal{O}(\log^{-1/2} K)) \frac{\rho \delta\log K}{K} I_K(F), \\
L_K(F) & \leq (1+ \mathcal{O}(\log^{-1/2} K)) \bigg(\frac{\rho \delta\log K}{K} \bigg)^2 I_K(F).
\end{align*}
Thus, we have 
\begin{align} \label{Sigma}
\Sigma \geq  \mathfrak{S}(1+\mathcal{O}(\delta)) N W^{-1}B^{-K} I_K(F), 
\end{align}
where 
\begin{align*}
\mathfrak{S} := \rho  \delta \log K - \mu' a - 3.99 \mu M  \binom{K/M}{2}\bigg(\frac{\rho\delta \log K}{K} \bigg)^2,
\end{align*}
if we pick $K$ large enough so that $\log^{-1/2} K < \delta.$ Choosing $\mu=1/L$ for some positive integer $L$ we observe that $\mu' =  (1+L)/2,$ the maximum (\ref{mu}) being obtained at $v= L$ and $v=1+L$. Define the quantity $X$ by $XM: = \rho \delta \log K.$ Then by using $3.99 \leq (M-1)/a$ we obtain
\begin{align*}
\mathfrak{S} &= XM - \frac{1+L}{2} a - 3.99   \frac{M}{L}\binom{K/M}{2}\bigg(\frac{XM}{K} \bigg)^2 \\
& \geq  XM - \frac{1+L}{2} a - \frac{M-1}{a} \frac{M}{L} \frac{K^2}{2M^2} \bigg(\frac{XM}{K} \bigg)^2 \\
&  =   XM - \frac{1+L}{2}a -  \frac{M-1}{a}\frac{X^2 M}{2L}  =  \frac{a}{2(M-1)}>0,
\end{align*}
for $X=aL/(M-1)$ and $L=M,$ requiring that $K$ is large enough so that $\rho < 1$ for this choice of $X$.
  \qed

\section{Proof of Theorem \ref{maint}} \label{mainsec}
Theorem \ref{maint} now follows by the same argument as in \cite[Section 6]{BFM}, using our Proposition \ref{strong} in place of \cite[Theorem 4.3]{BFM}; for this we need the modified Erd\"os-Rankin construction given by \cite[Lemma 5.2]{BFM} which states:
\begin{lemma} \label{erlemma} Let $K \geq 1$ and $\beta_K \geq \beta_{K-1} \geq \cdots \geq \beta_1 \geq 0.$ Then there is a real number $y(\bm{\beta},K)$ such that the following holds:

Let $x,y,z$ be any real numbers such that $x \geq 1$, $y \geq y(\bm{\beta},K)$, and
\begin{align*}
2y(1+(1+\beta_K)x)\leq 2z \leq y (\log_2 y)(\log_3 y)^{-1}.
\end{align*}
Let $\mathcal{Z}$ be any (possibly empty) set of primes such that for any $q \in \mathcal{Z}$ we have
\begin{align*}
\sum_{p\in \mathcal{Z}, \, p \geq q} 1/p \ll 1/q \ll 1/\log z.
\end{align*}
Then there is a set of integers $\{a_p:p\leq y, \, \, p \notin \mathcal{Z}\}$ and an admissible $K$-tuple $\{h_1,h_2,\dots,h_K\}$ such that
\begin{align*}
\{h_1,h_2,\dots,h_K\} &= ((0,z] \cap \Z) \setminus \bigcup_{p \leq y, \, p \notin \mathcal{Z}} \{m: m \equiv a_p \quad (p) \}, \\
P^+ \bigg( \prod_{1 \leq i < j\leq K} (h_j-h_i) \bigg) & \leq y,
\end{align*}
and for all $i=1,2,\dots,K$
\begin{align*}
h_i = \beta_i xy +y + \mathcal{O}\left( y e^{-\log^{1/4} y}\right).
\end{align*}
\end{lemma}

Given $\beta_1 \leq  \beta_2 \leq \beta_3 \leq \beta_4$ as in Theorem \ref{maint} and any sufficiently large $N$, we will apply the above lemma with 
\begin{align*}
x&:= 1/\epsilon, \quad \quad y:= \epsilon \log N, \quad \quad z:=y (\log_2 y)(2\log_3 y)^{-1}, \\
\bm{\beta}&:= \{\beta_1,\dots,\beta_1, \beta_2,\dots,\beta_2,\beta_3,\dots,\beta_3,\beta_4,\dots,\beta_4,\},
\end{align*}
where $\epsilon>0$ is sufficiently small and  each $\beta_i$ is repeated $K/4$ times for some sufficiently large $K \equiv 0 \,(4)$; by translation we may assume $\beta_1\geq 0.$  We let $\mathcal{Z}:= \{Z_{N^{4\epsilon}}\}$ if $Z_{N^{4\epsilon}} > 1,$ and $\mathcal{Z}=\emptyset$ otherwise (recall (\ref{Z}) for the definition of $Z_T$). The conditions of Lemma \ref{erlemma} are satisfied, so we get a set of integers $\{a_p: p\leq y, \, \, p \neq Z_{N^{4\epsilon}}\}$ and an admissible $K$-tuple $\HH$ such that
\begin{align}
\label{er}&\HH = ((0,z] \cap \Z) \setminus \bigcup_{p \leq \epsilon \log N, \, p \neq Z_{N^{4\epsilon}}} \{m: m \equiv a_p \quad (p) \}, \\ \nonumber
& P^+ \bigg( \prod_{1 \leq i < j\leq K} (h_j-h_i) \bigg) \leq \epsilon \log N,
\end{align}
such that there is a partition $\HH=\HH_1 \cup \HH_2 \cup\HH_3 \cup\HH_4$ into sets of equal sizes so that for all $i=1,2,3,4$ and for all $h \in \HH_i$
\begin{align*}
h = (\beta_i+\epsilon + o(1)) \log N.
\end{align*}
Let $b$ be an integer satisfying
\begin{align*}
b \equiv -a_p \quad (p) \quad \quad \text{for all} \quad \quad  p \leq \epsilon \log N, \, p \neq Z_{N^{4\epsilon}}.
\end{align*}
Then the assumptions of Proposition \ref{strong} are satisfied, so that the proposition yields two indices $1\leq i<j \leq 4$ and an integer $n \in [N,2N]$  with $n \equiv b \, (W)$ such that both  $n+\HH_i$ and $n+\HH_j$ contain a prime number. Furthermore, since $n \equiv b \, (W),$ by (\ref{er})  we have
\begin{align*}
\PP \cap (n,n+z] \subseteq n + \HH.
\end{align*}
Thus, for some $1\leq i<j \leq 4$, there are consecutive primes $p,q \in n+ \HH$ such that
\begin{align*}
p =(\beta_i+\epsilon + o(1)) \log N, \quad \quad \text{and} \quad \quad q = (\beta_j+\epsilon + o(1)) \log N.
\end{align*}
Since this holds for all sufficiently large $N$, we obtain that for some $1\leq i<j \leq 4$ we have $\beta_j-\beta_i \in \LL$.
 \qed

\section{A correction to the proofs of Lemmata \ref{s1} and \ref{s2}.}
The above text agrees with the published version of the article. Unfortunately there is a mistake in the proofs of Lemmata \ref{s1} and \ref{s2} (thanks to Jacques Benatar for pointing this out to me). Namely, in the remainder $r_d$, if we write
\begin{align*}
r_d = \sum_{\substack{d_1, \dots, d_K \\ d'_1,\dots d_K' \\ d_j=d_j'=d_\ell=d_\ell'=1}} \lambda_{d_1,\dots, d_K}  \lambda_{d_1',\dots, d_K'} \bigg( \sum_{\substack{N < n \leq 2N \\ n \equiv b \, (W) \\n \equiv -h_\ell \, (d) \\ n \equiv -h_i \, ([d_i,d_i'])}} 1_{\PP}(n+h_j) - g(d) \sum_{\substack{N < n \leq 2N \\ n \equiv b \, (W) \\ n \equiv -h_i \, ([d_i,d_i'])}} 1_{\PP}(n+h_j) \bigg),
\end{align*}
then the first sum in the brackets is empty if $(d,d_i)>1$ for some $i$ but the second sum is not empty. Note that this problem does not happen in our argument for $S_3$ (or in the proof of \cite[Lemma 6(iii)]{BFM}) where the Selberg sieve is used, thanks to the fact that the Selberg sieve wieghts are readily of the same form as the Maynard-Tao sieve weights. That is, the linear sieve we have used is not immediately compatible with the Maynard-Tao sieve. In this section we explain how to fix this issue. As is so often the case, the fundamental lemma of the sieve comes to the rescue. The idea is to handle small prime factors with  Selberg type sieve weights (in the spirit of the fundamental lemma of the sieve), so that in the linear sieve $g(d)$ and $r_d$ will be supported on numbers with no small prime factors so that the contribution from the part where $(d,d_i)>1$ is negligible.

Yet another problem is caused by the possible prime $Z_{N^{4\epsilon}}$. This is not a problem for the upper bounds of $S_1$ and $S_3$ but for the lower bound $S_2$ we cannot simply ignore $Z_{N^{4\epsilon}}$ as we have done above. This is resolved as follows. For any $y>1$ we define
\begin{align*}
P_0(y) := \prod_{\substack{p \leq y\\ p \nmid Z_{N^{4\epsilon}}}} p.
\end{align*}
In the original sum we write
\begin{align*}
S&:=\sum_{\substack{N < n \leq 2N \\ n \equiv b \, (W)}} 1_{\PP}(n+h_j)1_{\PP}(n+h_\ell)\nu_{\HH,j,\ell} (n)  \\
&\leq  \sum_{\substack{N < n \leq 2N \\ n \equiv b \, (W)}} 1_{\PP}(n+h_j)1_{(n+h_\ell,P_0(Z))=1}\nu_{\HH,j,\ell} (n) .
\end{align*}
We just have to note that the conclusion of Lemma \ref{chen} remains valid if we replace $P(Y),$ $P(Z)$, and $P(q)$ respectively by $P_0(Y),$ $P_0(Z)$, and $P_0(q)$.

Let $y_1:=\exp(\log^{1/3} N)$ and $y_2 := N^{\delta^2}$. Define the Selberg type sieve weights
\begin{align*}
\varrho(n):= \bigg( \sum_{\substack{e|(n,P_0(y_1)) }} \mu(e) G\bigg(\frac{\log e}{\log N}\bigg) \bigg)^2,
\end{align*}
where $G: [0,\infty) \to [0,1]$ is a smooth function supported on $[0,2\delta^2]$ and such that $G(u)=1$ for $u \in [0,\delta^2]$. That is, $G(\log e/\log N)= 1$ for $e \leq y_2$. 

Note that
\begin{align*}
1_{(n,P_0(y_1))=1} = \sum_{d| (n,P_0(y_1))} \mu(d) = \bigg( \sum_{d| (n,P_0(y_1))} \mu(d)\bigg)^2.
\end{align*} 
By definition $\varrho(n)$ is a sieve weight of very high level $y_2$ compared to the hight of shifting $y_1$, and for $(n,P_0(y_1))=1$ we have
\begin{align*}
\varrho(n) = G(0)^2 = 1 = 1_{(n,P_0(y_1))=1}.
\end{align*}
Hence, $\varrho(n)$ is equal to $1_{(n,P_0(y_1))=1}$ except on a very sparse set of integers, namely, integers $n$ which have a factor $d \geq y_2$ such that $d|P_0(y_1)$. Then $d$ has a factor $d_1$ such that $d_1 \in [y_2,y_1 y_2]$, so that we have shown the following.
\begin{lemma}
With the above notations, we have
\begin{align*}
1_{(n,P_0(y_1))=1} = \varrho(n) + O(\tau(n)^2 \psi(n;y_1,y_2))
\end{align*}
where $\psi(n;y_1,y_2)$ is the characteristic function of the set 
\begin{align*}
\{n: \, \exists d | n, \, d \in [y_2,y_1 y_2], \, d|P_0(y_1)\}.
\end{align*}
\end{lemma}
To bound the error term when using the above lemma, we require the following standard bound for the number of exceptionally smooth numbers.
\begin{lemma} \label{smoothlemma} For any $2 \leq z \leq y$ we have
\begin{align*}
\sum_{\substack{n \sim y \\ P^+(n) < z}} 1 \, \ll \, y e^{-u/2},
\end{align*}
where $u:= \log y/\log z.$
\end{lemma}

We are now ready to show the claimed bounds for $S_1$ and $S_2$. Letting
\begin{align*}
P(y_1,Y) := \prod_{y_1 < p \leq Y} p,
\end{align*}
we write
\begin{align*}
1_{(n+h_\ell,P_0(Y))=1} & = 1_{(n+h_\ell,P_0(y_1))=1} 1_{(n+h_\ell,P(y_1,Y))=1} \\
&= \varrho(n+h_\ell) 1_{(n+h_\ell,P(y_1,Y))=1}  + O(\tau(n+h_\ell)^2 \psi(n + h_\ell;y_1,y_2)).
\end{align*}
Note that for $S_1$ we could simply use $1_{(n+h_\ell,P_0(y_1))=1}  \leq \varrho(n)$, but for $S_2$ we need a lower bound. In either case the error term gives a contribution bounded by
\begin{align*}
\ll E(N) := \sum_{\substack{N< n \leq 2N \\ n \equiv b \, (W)}} \psi(n+h_\ell;y_1,y_2) \tau(n+h_\ell)^2 \nu_{\HH,j,\ell} (n) .
\end{align*}
By H\"older's inequality we get
\begin{align*}
E(N)^3 \ll \sum_{\substack{N/2< n \leq 4N \\ n \equiv b+h_\ell \, (W)}}  \tau(n)^{6} \sum_{\substack{N< n \leq 2N \\ n \equiv b \, (W)}}  \nu_{\HH,j,\ell} (n)^3 \sum_{\substack{N/2< n \leq 4N \\ n \equiv b+h_\ell \, (W)}} \psi(n;y_1,y_2)^3  \\
\ll \frac{N^2 \log^{O(1)} N}{W^2} \sum_{\substack{N/2< n \leq 4N \\ n \equiv b+h_\ell \, (W)}} \psi(n;y_1,y_2).
\end{align*}
By assumptions we have $(b+h_\ell,W)=1$ so that $(n,W)=1$, and we get by Lemma \ref{smoothlemma}
\begin{align*}
 \sum_{\substack{N/2< n \leq 4N \\ n \equiv b+h_\ell \, (W)}} \psi(n;y_1,y_2) \ll \sum_{\substack{y_2 \leq d \leq y_1 y_2 \\ d|P_0(y_1) \\ (d,W)=1}} \sum_{\substack{N/2d < n  \leq 4N/d \\ dn \equiv b + h_\ell \, (W)}} 1 \ll \frac{N}{W}\sum_{\substack{y_2 \leq d \leq y_1 y_2 \\ P^+(d) \leq y_1 }} \frac{1}{d} \ll \frac{N}{W} \exp(-(\log^{1/2} N)).
\end{align*}
Hence, for any $C>0$ we have
$E(N) \ll_C  (\log^{-C}N) N /W$,
which is sufficient.

Therefore, it suffices to prove the claimed bounds in Lemmata  \ref{s1} and \ref{s2} for the modified sums
\begin{align*}
S'_1 & := \sum_{\substack{N < n \leq 2N \\ n \equiv b \, (W)}} 1_{\PP}(n+h_j) 1_{(n+h_\ell,P(y_1,Y))=1}\varrho(n) \nu_{\HH,j,\ell} (n)  \quad \quad \text{and}   \\
S'_2 & := \sum_{Y < p \leq Z} \sum_{\substack{N < n \leq 2N \\ n \equiv b \, (W) \\ p | n+h_\ell}} 1_{\PP}(n+h_j)1_{(n+h_\ell,P(y_1,Y))=1} \varrho(n)\nu_{\HH,j,\ell} (n).
\end{align*}
Notice here that crucially $\varrho(n) \geq 0,$ so that we can still apply the linear sieve. We now show how to handle $S'_1$, the details are the same for $S'_2$.
\begin{lemma} We have
\begin{align*}
S'_1 \leq \frac{F_{\text{\emph{lin}}}(1/(2\alpha)) + \mathcal{O}(\delta)}{\alpha e^\gamma}  \frac{N}{W} B^{-K}L_K(F)
\end{align*}
\end{lemma}
Let 
$r_d$ be defined by the equation
\begin{align} 
 \sum_{\substack{N < n \leq 2N \\ n \equiv -h_\ell \, (d)}} 1_{\PP}(n+h_j) 1_{n \equiv b \, (W)}\varrho(n)\nu_{\HH,j,\ell} (n)  = g(d)  \sum_{\substack{N < n \leq 2N}} 1_{\PP}(n+h_j) 1_{n \equiv b \, (W)} \varrho(n)\nu_{\HH,j,\ell} (n)  +r_d,
\end{align}
where $g(d)$ is a multiplicative function, supported on square-free integers, defined by
\begin{align*}
g(p) := \begin{cases} \frac{1}{p-1}, & \text{if} \,  p >y_1 \\
0, & \text{if} \,  p \leq y_1.
\end{cases}
\end{align*}
To handle the error term in the linear sieve upper bound, we write for $(d,P(y_1))=1$
\begin{align*}
r_d &=  \sum_{\substack{N < n \leq 2N \\ n \equiv -h_\ell \, (d)}} 1_{\PP}(n+h_j) 1_{n \equiv b \, (W)} \varrho(n) \nu_{\HH,j,\ell} (n)  - g(d)  \sum_{\substack{N < n \leq 2N}} 1_{\PP}(n+h_j) 1_{n \equiv b \, (W)} \varrho(n)\nu_{\HH,j,\ell} (n)  \\
&= \sum_{e,e' | P_0(y_1)} \delta_e \delta_{e'} \sum_{\substack{d_1, \dots, d_K \\ d'_1,\dots d_K' \\ d_j=d_j'=d_\ell=d_\ell'=1}} \lambda_{d_1,\dots, d_K}  \lambda_{d_1',\dots, d_K'} \\
& \hspace{100pt} \bigg( \sum_{\substack{N < n \leq 2N \\ n \equiv b \, (W) \\n \equiv -h_\ell \, (d [e,e']) \\ n \equiv -h_i \, ([d_i,d_i'])}} 1_{\PP}(n+h_j) - g(d) \sum_{\substack{N < n \leq 2N \\ n \equiv b \, (W) \\ n \equiv -h_\ell \, ( [e,e'])\\ n \equiv -h_i \, ([d_i,d_i']) }} 1_{\PP}(n+h_j) \bigg),
\end{align*}
where we have defined $\delta_e := \mu(e) G(\log e / \log N)$. Again, in the first sum (in the brackets) we have $(d,d_i)=1$. For the second sum we note that if $(d,d_i) > 1$ for some $i$, then  by $(d,P(y_1))=1$ we get $(d,d_i) > y_1$, and the contribution from this can be bounded trivially. Indeed, the  part where $(d,d_i)=c > y_1$ for some $i$ gives a contribution bounded by
\begin{align*}
& \ll \sum_{\substack{ d \leq N^{1/2-4\delta}
\\ (d,P(y_1))=1}} g(d) \sum_{\substack{c|d \\ c> y_1}} \sum_{\substack{N < n \leq 2N \\ n \equiv b \, (W) \\ n \equiv h_i \, (c)}} \bigg( \sum_{\substack{d_1, \dots, d_K \\ d_i | n+h_i \\ d_j=d_\ell =1}} |\lambda_{d_1, \dots, d_K}| \bigg)^2 \\
& \ll \frac{N^{1/2} \log^{O(1)} N}{W^{1/2}} \bigg(\sum_{\substack{ d \leq N^{1/2-4\delta}
\\ (d,P(y_1))=1}} g(d) \sum_{\substack{c|d \\ c> y_1}} \sum_{\substack{N < n \leq 2N \\ n \equiv b \, (W) \\ n \equiv h_i \, (c)}} 1 \bigg)^{1/2}  \\
& \ll \frac{N \log^{O(1)} N}{W} \bigg(\sum_{\substack{ d \leq N^{1/2-4\delta}
\\ (d,P(y_1))=1}} g(d) \sum_{\substack{c|d \\ c> y_1}} \frac{1}{c} \bigg)^{1/2}  \ll  \frac{N \log^{O(1)} N}{W y_1^{1/2}} \ll_C \frac{N }{W \log^C N}.
\end{align*}
Hence, in the error term $r_d$ we may restrict to $(d,d_i)=1$ for all $i$ and use the same argument as in proof of Lemma \ref{s1} to get a sufficient bound for $\sum_{d} |r_d|$.

For the main term we now have show that
\begin{align}  \nonumber
\sum_{\substack{N < n \leq 2N}} 1_{\PP}(n+h_j) 1_{n \equiv b \, (W)} &\varrho(n)\nu_{\HH,j,\ell} (n)  \\ \label{mainevaluation} & = (1+o(1)) \prod_{\epsilon \log N < p \leq y_1}\bigg(1-\frac{1}{p} \bigg) \frac{N}{\phi(W) \log N}  B^{-K+2}L_K(F), 
\end{align}
where the difference compared to (\ref{main}) is the coefficient $\varrho(n)$, which will result in the (expected) extra factor  $\prod_{\epsilon \log N  < p \leq y_1}(1-p^{-1})$ on the right-hand side. This evaluation follows from a similar argument as in the proof of \cite[Lemma 4.6 (iii)]{BFM}. The only difference is that we will need a slightly more general version of \cite[Lemma 4.5]{BFM} (or \cite[Lemma 4.1]{polymath}), due to the fact that in $\varrho(n)$ the variable $e$ is restricted to $e|P_0(y_1)$. More precisely, we need the following (which is applied with $k=K-2$).
\begin{lemma} Let $y_1 := \exp(\log^{1/3}N)$. Let $F_0,\dots, F_k, G_0, \dots G_k : [0,\infty) \to \R$ be fixed smooth compactly supported functions.  Denote $B:=  (\log N) \phi(W)/W$ Then
\begin{align*}
 \sideset{}{'}\sum_{\substack{d_1,\dots,d_k \\ d_1',\dots,d_k' \\ e,e' | P(y_1)}} \frac{\mu(e)\mu(e')}{[e,e']} F_0 \bigg( \frac{\log e}{\log N}\bigg) G_0 \bigg( \frac{\log e'}{\log N}\bigg) \prod_{j=1}^{k} \frac{\mu(d_j) \mu(d_j')}{[d_j,d_j']} F_j \bigg( \frac{\log d_j}{\log N}\bigg) G_j \bigg( \frac{\log d'_j}{\log N}\bigg) \\
 =(c+o(1))B^{-k} \prod_{\epsilon \log N < p \leq y_1} \bigg( 1-\frac{1}{p}\bigg),
\end{align*}
where $\sideset{}{'}\sum$ denotes restriction that $[e,e'],[d_1,d_1'],\dots,[d_k,d_k'], Z_{N^{2\epsilon}} W$ are pairwise coprime, and 
\begin{align*}
c = F_0(0)G_0(0)\prod_{j=1}^k \int F_j(t_j) G_j(t_j) d t_j.
\end{align*}
The same holds if $[e,e']$  and $[d_j,d_j']$ are replaced by $\phi([e,e'])$  and $\phi([d_j,d_j'])$.
\end{lemma}
The proof of this lemma follows by the same argument as the proof of \cite[Lemma 4.1]{polymath}, and we refer to the notations used there (with the exception that $x$ there corresponds to $N$ here). We obtain that the left-hand side is (up to a negligible error term)
\begin{align*}
\int_{-\sqrt{\log N}}^{\sqrt{\log N}} \cdots \int_{-\sqrt{\log N}}^{\sqrt{\log N}} K(\xi_0, \dots, \xi_k, \xi'_0, \dots, \xi'_k) \prod_{j=0}^k f_j(\xi_j) g_j(\xi'_j) d\xi_1 \cdots d\xi_k d\xi'_1 \cdots d\xi_k',
\end{align*}
where
\begin{align*}
K(\xi_0, \dots, \xi_k, \xi'_0, \dots, \xi'_k)=(1+o(1)) L(\xi_0,\xi_0')K(\xi_1, \dots, \xi_k, \xi'_1, \dots, \xi'_k) 
\end{align*}
 with $K(\xi_1, \dots, \xi_k, \xi'_1, \dots, \xi'_k)$ as in \cite[proof of Lemma 4.1]{polymath} and 
\begin{align*}
L(\xi_0,\xi_0') = \frac{\zeta_0 (1 + (2+i\xi_0 + i\xi_0')/\log N)}{\zeta_0 (1 + (1+i\xi_0 )/\log N) \zeta_0 (1 + (1+i\xi_0')/\log N)}
\end{align*}
where
\begin{align*}
\zeta_0(s) := \prod_{\epsilon \log N < p \leq y_1} (1-p^{-s})^{-1}.
\end{align*}
Since $|\xi_j| \leq \sqrt{ \log N}$, we have
\begin{align*}
&\prod_{\epsilon \log N < p \leq y_1} \bigg(1-\frac{1}{p^{1+(1+i \xi_0)/\log N}} \bigg) =\prod_{\epsilon \log N < p \leq y_1} \bigg(1-\frac{1 + O\left(\frac{\log p}{\sqrt{\log N}}\right)}{p} \bigg) \\
&=  \exp\bigg( O\bigg(\sum_{p\leq y_1 }  \frac{\log p}{p \sqrt{\log N}}\bigg)\bigg)\prod_{\epsilon \log N < p \leq y_1} \bigg(1-\frac{1}{p} \bigg) =\exp\bigg(O \bigg( \frac{\log y_1}{\sqrt{\log N}}\bigg) \bigg) \prod_{\epsilon \log N < p \leq y_1} \bigg(1-\frac{1}{p} \bigg)  \\
&= (1+o(1)) \prod_{\epsilon \log N < p \leq y_1} \bigg(1-\frac{1}{p} \bigg)
\end{align*}
by using $y_1 = \exp(\log^{1/3} N)$. The same holds with $1+i \xi_0$ replaced by $1+i \xi'_0$ or $2+i \xi_0 + i \xi'_0$ Hence,
\begin{align*}
L(\xi_0, \xi'_0) = (1+o(1)) \prod_{\epsilon \log N < p \leq y_1} \bigg(1-\frac{1}{p} \bigg),
\end{align*}
and the remainder of the proof is essentially the same as \cite[proof of Lemma 4.1]{polymath} since by definition
\begin{align*}
\int \int f_0(t_0) g_0( t_0') dt_0 dt'_0 = F(0)G(0).
\end{align*}
Similarly as in \cite[Lemma 4.1]{polymath}, the error terms from the factors 1+o(1)  are negligible by the rapid decay of $f_j$ and $g_j$.

Using the above Lemma we get (\ref{mainevaluation}), so that the linear sieve upper bound yields
\begin{align*}
&S'_1 \leq  
(F_{\text{\emph{lin}}}(1/(2\alpha)) +O(\delta)) \bigg(\prod_{p\leq Y} (1- g(p)) \bigg) \bigg(\prod_{\epsilon \log N < p \leq y_1} \bigg( 1-\frac{1}{p}\bigg) \bigg) \frac{N}{W} B^{-K+1} L_K(F) \\
&=(F_{\text{\emph{lin}}}(1/(2\alpha)) +O(\delta)) \bigg(\prod_{y_1 < p\leq Y} \bigg(1-\frac{1}{p-1} \bigg) \bigg) \bigg(\prod_{\epsilon \log N < p \leq y_1} \bigg( 1-\frac{1}{p}\bigg) \bigg) \frac{N}{W} B^{-K+1} L_K(F) \\
&=(F_{\text{\emph{lin}}}(1/(2\alpha)) +O(\delta))  \bigg(\prod_{\epsilon \log N < p \leq Y} \bigg( 1-\frac{1}{p}\bigg) \bigg) \frac{N}{W} B^{-K+1} L_K(F)  \\
&=(F_{\text{\emph{lin}}}(1/(2\alpha)) +O(\delta))  \bigg(\prod_{ p \leq Y} \bigg( 1-\frac{1}{p}\bigg) \bigg) \frac{N}{\phi(W)} B^{-K+1} L_K(F)  \\
&=\frac{F_{\text{\emph{lin}}}(1/(2\alpha)) + O(\delta)}{\alpha e^\gamma}  \frac{N}{W} B^{-K}L_K(F)
\end{align*}
by Merten's theorem since $Y=N^\alpha$.
\bibliography{limitp}
\bibliographystyle{abbrv}

\end{document}